\documentclass[12pt]{amsart}
\usepackage[english]{babel}
\usepackage{amsthm, amsmath, amssymb, amsfonts, amscd, geometry}
\usepackage[svgnames]{xcolor} 
\usepackage[colorlinks,citecolor=DarkViolet,pagebackref=true,urlcolor=DarkViolet, bookmarksdepth=3, linkcolor=DarkCyan]{hyperref} 
\usepackage{kpfonts} 

\usepackage{stmaryrd}
\usepackage{listings}

\usepackage{graphicx}
\usepackage{tikz-cd}
\usepackage{indentfirst}
\usepackage{mathrsfs}
\usepackage{accents}
\usepackage{xcolor}

\usepackage{listings}
\usepackage{color}

\definecolor{darkspringgreen}{rgb}{0.09, 0.45, 0.27}
\definecolor{emerald}{rgb}{0.31, 0.78, 0.47}
\definecolor{indiagreen}{rgb}{0.07, 0.53, 0.03}
\definecolor{bulgarianrose}{rgb}{0.28, 0.02, 0.03}
\definecolor{brandeisblue}{rgb}{0.0, 0.44, 1.0}
\definecolor{bananayellow}{rgb}{1.0, 0.88, 0.21}

\newcommand{\done}[1]{}

\def\dim{\operatorname{dim}}

\def\exp{\operatorname{exp}}

\def\mod{\operatorname{mod}}
\def\spec{\operatorname{Spec}}

\def\div{\operatorname{div}}
\def\cone{\operatorname{cone}}

\def\Gr{\operatorname{Gr}}
\def\ord{\operatorname{ord}}

\def\ord{\operatorname{ord}}
\def\dim{\operatorname{dim}}

\def\conv{\operatorname{conv}}

\def\Z{\mathbb{Z}}
\def\Q{\mathbb{Q}}
\def\R{\mathbb{R}}
\def\C{\mathbb{C}}

\def\F{\mathcal{F}}

\def\P{\mathbb{P}}
\def\O{\mathcal{O}}

\def\T{\mathbb{T}}
\def\X{\mathscr{X}}
\def\Y{\mathscr{Y}}

\def\Zv{\mathscr{Zv}}

\def\BCOV{\tau_{BCOV}}

\def\H{\mathcal{H}}
\def\F{\mathcal{F}}

\newtheorem{theorem}{Theorem}[section]
\newtheorem{lemma}[theorem]{Lemma}
\newtheorem{proposition}[theorem]{Proposition}

\newtheorem{corollary}[theorem]{Corollary}

\newtheorem{remark}[theorem]{Remark}
\let\oldremark\remark
\renewcommand{\remark}{\oldremark\normalfont}

\newtheorem{definition-proposition}{Definition-Proposition}[section]
\newtheorem{definition-lemma}[theorem]{Definition-Lemma}

\title{Genus one mirror symmetry for intersection of two cubics in $\mathbb{P}^5$}
\author{Dennis Eriksson, Mykola Pochekai}

\address{Dennis Eriksson \\ Department of Mathematics \\ Chalmers University of Technology and  University of Gothenburg}
\email{dener@chalmers.se}

\address{Mykola Pochekai \\ Center for Geometry and Physics \\ Institute for Basic Science}
\email{pochekai@ibs.re.kr}

\thanks{The first author is supported by the Swedish Research Council, VR grant 2021-03838 "Mirror symmetry in genus one". This constitutes part of the second author's thesis.}

\begin{document}


\begin{abstract}
This paper establishes BCOV-type genus one mirror symmetry for the intersections of two cubics in  $\mathbb{P}^5$. The proof applies previous constructions of the mirror family by the second author and computations of genus one Gromov-Witten invariants by A. Popa. The approach adapts the strategy used for hypersurfaces, as developed by the first author and collaborators, but addresses the distinct geometry involved. A key feature is a systematic usage of toric techniques and related computer aided calculations to determine seemingly otherwise inaccessible invariants. 
\end{abstract}

\maketitle  
\setcounter{tocdepth}{1}

\tableofcontents

\section{Introduction}

\begin{sloppypar}
Mathematical mirror symmetry originates from the string theory work in \cite{CdGP}, which proposed a conjectural framework for understanding genus zero Gromov--Witten invariants. This was mathematically proven in the case of complete intersections in toric varieties in the work \cite{givental}. 
\end{sloppypar}

In the seminal work \cite{BCOV} a higher genus Gromov--Witten mirror symmetry statement was suggested. For quintic threefolds, the conjecture was affirmatively answered for the genus one Gromov--Witten invariants in \cite{FLY}, relying on work of \cite{Zin}. It was generalized to higher dimensional projective hypersurfaces in \cite{EFiMM3}, where the conjecture was also made more precise.

\subsection{Degenerations, mirrors and distinguished sections} 

Without loss of generality for the considerations in this paper, we consider a degeneration of $\Y \to \Delta^\times$ Calabi--Yau threefolds, over a punctured disc. We also suppose that the complex moduli of a general fiber is one-dimensional. Mirror symmetry is expected to work for large complex structure limits, whose precise mathematical definition can sometimes vary. Under the assumption of the complex moduli being one-dimensional, this is usually taken to mean that the monodromy $T$ is maximally unipotent, meaning the Jordan blocks of the monodromy have the maximal size, meaning in this case that 
\begin{displaymath}
    (T-1)^3 \neq 0.
\end{displaymath} This implies in this case the a priori stronger condition that the middle Hodge structure is Hodge--Tate (cf. \cite[p. 80]{CK}). We assume this and further assume the non-middle cohomologies also have Hodge--Tate limit mixed Hodge structures. These are often even trivial, and seems to hold in most cases for formal reasons. One hopes that in this type of situation, there exists a mirror Calabi--Yau manifold $X$, with some standard expected properties.  

 These assumptions imply that the degeneration $\Y \to \Delta^\times$ is strongly unipotent in the sense of \cite[Definition 6.10]{EFiMM3}. As explained in \cite[6.2]{EFiMM3}, it follows that in a small neighborhood of the origin  there are canonical coordinates $\psi \mapsto Q(\psi)$ of the unit disc, and distinguished sections of the (determinants of the) Hodge bundles $H^{p,q}(\Y_\psi)$ adapted to the corresponding weight filtration, as written in \cite[Theorem 5.9]{EFiMM3}. The distinguished sections are uniquely determined up to a constant  independent of $\psi.$

\subsection{The BCOV conjecture in our setting}

To set up the formulation and the conjecture and result, for a Calabi--Yau 3-fold $Y$, one considers the expression
\begin{equation}\label{eq:BCOV-intro}
\BCOV(Y) = C(Y,h_Y)\prod_{p \geq 0} \tau(Y, \Omega^p_{\Y}, h_Y)^{p (-1)^p},
\end{equation}
where $\tau(Y, \Omega^p_Y,h_Y)$ is the holomorphic Ray--Singer analytic torsion with respect to some auxiliary K\"ahler metric $h_Y$, and $C(Y,h_Y)$ is some correcting factor which depends explicitly on $h_Y$. This was introduced in \cite{FLY}, but the above expression is inspired by \cite{EFiMM2}. As suggested by the notation in \eqref{eq:BCOV-intro}, the resulting expression does not depend on $h_Y$, and it is called the BCOV invariant. In holomorphic families of Calabi--Yau manifolds $\Y \to S$, the function $S \ni \psi \mapsto \BCOV(\Y_{\psi})$ defines a smooth real-valued function.  

In the setting of mirror symmetry of a family of Calabi--Yau 3-folds with one-dimensional complex moduli, $\Y \to  \Delta^\times$, with maximally unipotent monodromy, denote by $\widetilde{\eta}_{p,q}$ the distinguished sections of $\det H^{p,q}(\Y_\psi)$. Then the conjecture of \cite{EFiMM3} can then be stated, for families satisfying the  above assumptions, as follows: up to a constant, there is an equality
\begin{equation}\label{EFiMM:BCOV-conjecture}
\BCOV(\Y_\psi) = |\exp(- F_{1,B})|^4 \cdot ||\widetilde{\eta}_{3,0}||_{L^2}^{\chi/6 } \cdot \prod_{p,q} ||\widetilde{\eta}_{p,q}||^{2 p}_{L^2},
\end{equation}
for a certain holomorphic function $\exp(-F_{1,B}(\psi)).$ The expression $\|\cdot \|_{L^2}$ is the $L^2$-norm of the corresponding cohomology group, coming from Hodge theory, which for primitive classes $\eta$ in the middle cohomology is given by
\begin{equation}\label{eq:L2primitive}\|{\eta}\|^2_{L^2} = i \int_{\Y_\psi} {\eta} \wedge \overline{{\eta}}.
\end{equation}
The expression in \eqref{EFiMM:BCOV-conjecture} a priori depends on a K\"ahler metric, whose dependence is part of the indeterminate constant. 
Then, if $Q(\psi)$ denotes the canonical coordinate alluded to above, we should have 
\begin{equation}\label{eq:F1AB} F_{1,B}(\psi) = F_{1,A}(Q) :=N^{0}_1(X) \log Q+ \sum_{d>0} N^{d}_1(X) Q^d
\end{equation}
where $N^{d}_1(X)$ denotes the genus one Gromov--Witten invariants of a mirror manifold $X$. Notice that the mirror symmetry predicts the K\"ahler cone of the mirror is one-dimensional since we assume the complex moduli is one-dimensional, and hence the Gromov--Witten invariants are indexed by positive integers. See \eqref{definition:F1A} for the definition of the special case $d=0$.

\subsection{Intersection of two cubics in $\P^5$} In the paper we will prove the genus one mirror symmetry statement for Calabi--Yau varieties that are smooth complete intersections of two cubics in $\P^5$. Such an intersection will be denoted by $X_{3,3}$. In \cite[Theorem 1.1]{Poch1} a crepant desingularizations of the mirror dual family $\Y_{\psi}, \psi \in \P^1$ of $X_{3,3}$ has been constructed. The alternative constructions of the mirror family can be found in \cite{rossi} and \cite{malter}. These families are singular and cannot be used for the techniques at hand to study the conjecture in \eqref{EFiMM:BCOV-conjecture}. All the mentioned constructions either originate from or are fundamentally close to the Batyrev–-Borisov construction \cite{BB} of mirrors for Calabi–Yau complete intersections in toric varieties. The precise geometry of $\Y_{\psi}, \psi \in \P^1$ is crucial for several of our computations, including those in Proposition \ref{th:euler-characteristics-of-Y0} and Proposition \ref{th:holomorphic-euler-characteristics-of-Y0}. We remark that $h^{p,3-p}=1$ for $p \geq 0$, so, in particular the dimension of the complex moduli is $h^{1,2} = 1$.

The main theorem of our paper is the next one:
\begin{theorem} \label{th-A}
    The family $\Y \to \Delta^\times$ over a punctured disc around infinity is strongly unipotent and the conjecture in \eqref{EFiMM:BCOV-conjecture} holds. More precisely, the only non-trivial distinguished sections are sections $\widetilde{\eta}_p$ of $H^{3-p,p} (\Y_{\psi})$, and the next identity holds in a neighborhood of infinity, up to constant:
    $$\BCOV(\Y_{\psi}) =  | \exp(-F_{1,A}) |^4 \cdot || \widetilde{\eta_0} ||^{30}_{L^2}\cdot  || \widetilde{\eta_1} ||^{4}_{L^2} \cdot || \widetilde{\eta_2} ||^{2}_{L^2}.$$
    Here $F_{1,A}$ is the generating series in \eqref{eq:F1AB} for the Gromov--Witten invariants of $X_{3,3}$ and the $L^2$-norm is given by the formula in \eqref{eq:L2primitive}.
\end{theorem}

\subsection{Overview of strategy} We will briefly describe the proof strategy. The general approach is adapted from the hypersurface case developed in \cite{EFiMM3}, but there are several new challenges in the $X = X_{3,3}$ case. The basic strategy involves explicitly computing everything, substituting the results into the corresponding formulas, and verifying that the identity holds. The detailed steps are as follows:
\begin{itemize}
    \item In Section \ref{sec-mirror-geometry} we recall and study the geometry of the explicitly constructed mirror family. It has the structure of an intersection of two hypersurfaces in a toric variety, and is naturally fibered over $\mathbb{P}^1.$ There are three classes of singular fibers: 
    \begin{enumerate}
        \item $\infty$, which is the MUM-point.
        \item $\psi$ with $\psi^6=1$, which are fibers admitting a single ordinary double point singularity.
        \item 0, which is a so called $K$-point.
    \end{enumerate}
    We find sections $\eta_0, \ldots, \eta_3$ which are trivializations of the de Rham bundle $\H^3$. 
    \item We apply the arithmetic R iemann--Roch theorem from Arakelov geometry, to compute the quantity $\BCOV$ explicitly up to a rational function $\phi$ on $\P^1$, as shown in Theorem \ref{th:ARR}. 
    Because of the special geometries involved, we manage to conclude that $$\BCOV(\psi) = \left| \phi \right| \|\eta_0\|_{L^2}^{30} \cdot \| \eta_1 \|_{L^2}^4 \cdot  \| \eta_2\|_{L^2}^2.$$
    \item We consider the above  as a formula for $|\phi|$, and determine it by studying the asymptotic behavior of all its defining terms. This builds upon Schmid's estimates for Hodge-type norms and the asymptotic behaviors of the BCOV invariant from \cite{EFiMM2}.
    \item We determine the periods of the family, and use them to find an adapted basis $\widetilde{\eta}_p$ (cf. \textsection{\ref{subsec:adaptedbasis}}) yielding a formula of the type \eqref{EFiMM:BCOV-conjecture} and hence a computation of $F_{1,B}$. 
    \item Compute the genus one Gromov-Witten invariants $N_1^d(X_{3,3})$ explicitly and express it in terms of, as done in \cite{Popa}, the powers series:
    $$F_1^A(Q) = N_1^0(X_{3,3}) \log Q + \sum_{d \geq 0} N_1^d(X_{3,3}) Q^d.$$
    This can be directly compared with the function $F_{1,B}$ from the previous step, and one concludes.

\end{itemize}

\section{The mirror family and its geometry}\label{sec-mirror-geometry}
We will write $=$ between two objects (rings, varieties, or modules) if they are isomorphic and the isomorphism is either canonical or its choice is clear from the context. On the other hand, we will use the symbol $\cong$ between two objects if they are abstractly isomorphic. By variety we mean reduced (not necessarily irreducible) finite type $\C$-scheme. When we write $H^k(X,\Z)$ we mean non-torsion part of the corresponding abelian group. 

We will briefly summarize the paper \cite{Poch1}, which introduces the main object of our investigation and provides several useful technical results about it. The object of the investigation is the Batyrev-Borisov mirror dual family to the smooth intersection of two cubical hypersurfaces in $\P^5$ and its crepant desingularizations. 

\subsection{Batyrev-Borisov mirror dual}
Let $N = \{(x_1, \ldots, x_6) \in \mathbb{Z}^6 : x_1 + \cdots + x_6 = 0\}$ be a lattice of rank $5$, and let $ N_{\mathbb{R}} = N \otimes_{\mathbb{Z}} \mathbb{R}$ be a realification of aforementioned lattice. The lattice $N$ will play the role of lattice of one-parameter subgroups of the open torus $\T_N := N \otimes \C^{\times}$. The toric variety constructed from the fan $\Sigma$ will be denoted by $\P_{\Sigma}$. The dual lattice
$$M := N^{\vee} = \frac{\Z^6}{\Z \{e_1 + ... + e_6\}}$$
will play a role of the character lattice of the torus $\T_N$. If $v_1, \ldots, v_r \in N$ is a set of vectors, we use the notation
$$
\operatorname{cone} \{v_1, \ldots, v_r\} := \{\alpha_1 v_1 + \cdots + \alpha_r v_r : \alpha_1, \ldots, \alpha_r \in \mathbb{R}\} \subset N_{\mathbb{R}}
$$
to denote the cone generated by the vectors $v_1, \ldots, v_r$. Throughout the paper, by a cone, we mean a subset $\sigma \subset N_{\R}$ that has a presentation $\sigma = \cone \{v_1,...,v_k\}$ where $v_i \in N, 1 \leq i \leq k$. Additionally, all cones in our paper will be well-pointed, meaning that for any $v \in N_{\R}$, if both $v \in \sigma$ and $-v \in \sigma$, then $v = 0$.

By a fan $\Sigma$ with generators in the lattice $N$, we mean a nonempty collection of cones that is closed under finite intersections and the operation of taking faces. We will use $\Sigma(1)^{gen}$ to denote the set of ray generators of the fan $\Sigma$. Given a fan $\Sigma$ with generators in the lattice $N$, we can associate an algebraic variety $\P_{\Sigma}$ with it. The following proposition introduces a specific fan $\Sigma$ and its associated variety $\P_{\Sigma}$, which will serve as the ambient space for the mirror dual family.

\begin{proposition} \cite[Theorem 3.4]{Poch1} \label{th:dual-toric-data}
    There is a complete fan $\Sigma$ with generators in lattice $N$ which has the following set of ray generators:
    $$\Sigma(1)^{\text{gen}} = \{u_1, \dots, u_6, v_1, ..., v_6\},$$
    $$u_i = -e_1 - e_2 - e_3 + 3e_i, v_i = -e_4 - e_5 - e_6 + 3e_i.$$
    The associated toric variety \(\mathbb{P}_{\Sigma}\) is a Gorenstein Fano toric variety.   
\end{proposition}

For any $\mathbb{T}_N$-invariant Weil divisor $D$ of the normal toric variety $\P_{\Sigma}$, the vector space of global sections $\Gamma(\P_{\Sigma}, \mathcal{O}(D))$ can be identified with the vector space generated by the integral points of a fundamental polytope $P_D$:
$$\Gamma(\P_{\Sigma}, \mathcal{O}(D)) = \bigoplus_{m \in P_D \cap M} \mathbb{C} \{t^m\},$$
as stated in \cite[Proposition 4.3.3]{CLS}. The family which we aim to construct will be the complete intersection of the next two divisors:
$$D_{\nabla_1} = u_1 + ... + u_6, D_{\nabla_2} = v_1 + ... + v_6,$$ with the associated fundamental polytopes:
$$\nabla_1 = \conv\{0,e_1,e_2,e_3\}, \nabla_2 = \conv\{0,e_4,e_5,v_6\}.$$
The computation of the fundamental polytopes is presented in \cite[Theorem 3.4]{Poch1}. Thus, the global sections of those divisors can be described in the following way:
\begin{equation} \label{eq:toric-inside-global-sections}
    \Gamma(\P_{\Sigma}, \mathcal{O}(D_{\nabla_1})) = \mathbb{C} \{1, t_1, t_2, t_3\}, \Gamma(\P_{\Sigma}, \mathcal{O}(D_{\nabla_2})) = \mathbb{C} \{1, t_4, t_5, t_6\}.
\end{equation}

We will briefly remind the reader about a Cox ring and a Cox space. A detailed account of these concepts can be found in \cite[Chapter 5]{CLS}. A Cox ring is called a total coordinate ring in the source, a Cox space in the source has no special name, and is instead denoted by $\C^{K(1)} - Z(K)$. Let $K$ be a complete fan with ray generators in the lattice $N$. By the Cox ring, we mean the $\operatorname{Cl}(\P_{K})$-graded ring $\C[\rho : \rho \in K(1)^{gen}]$. By the Cox space we mean the space $\C^{K(1)} - Z(K)$, where $Z(K)$ is an irrelevant ideal, which is defined as 
$$Z(K) = \bigcup_{\substack{C \text{ is a } \\ \text{primitive collection}}} \bigcap_{r \in C} V(r).$$
A subset $C \subset K(1)$ is called a primitive collection if: 
\begin{itemize}
    \item $C \not\subset \sigma(1)$ for all $\sigma \in K$,
    \item For every proper subset $C' \subset C$, there exists $\sigma \in K$ such that $C' \subset \sigma(1)$.
\end{itemize}

When the fan $K$ is complete and smooth, the variety $\mathbb{P}_{K}$ can be presented as a geometric quotient $\mathbb{C}^{K(1)} - Z(K) \to \mathbb{P}_{K}$, this is the content of \cite[Theorem 5.1.11]{CLS}, that is why Cox space is worth considering. 

We will make no distinction between the ray generator $r \in K(1)^{gen}$, the corresponding coordinate function on the Cox space $r : \C^{K(1)} - Z(K) \to \C$, and the element of the Cox ring $r \in \C[\rho : \rho \in K(1)^{gen}]$, as all three are naturally identifiable.  One can emulate a lot of techniques from projective geometry if one pretends that the variety $\P_{K}$ is the projective space $\P^n$, the associated Cox ring $\C[\rho : \rho \in K(1)^{gen}]$ is the graded ring of homogeneous coordinates of the projective space $\C[x_0,...,x_n]$, and Cox space $\C^{K(1)} - Z(K)$ is the space $\C^{n+1} - \{0\}$.


We can describe $\Gamma(\mathbb{P}_{K}, \mathcal{O}(D))$ as a vector subspace of the Cox ring in the following way:
$$\Gamma (\mathbb{P}_{K}, \mathcal{O}(\sum_{\rho\in K(1)^{gen}} a_{\rho} D_{\rho})) = \bigoplus_{m \in M \cap P_{D}} \mathbb{C} \left\{\prod_{r \in K(1)^{gen}} r^{(m,r) + a_r}\right\}, $$
this is \cite[Proposition 5.4.1]{CLS}. In the case of the fan $K = \Sigma$ mention in the Proposition \ref{th:dual-toric-data}, we have:
$$\Gamma (\mathbb{P}_{\Sigma}, \mathcal{O}(D_{\nabla_1})) =  \left\{\alpha_0 u_1 ... u_6 + \alpha_1 u_1^3 v_1^3 + \alpha_2 u_2^3 v_2^3 + \alpha_3 u_3^3 v_3^3 : (\alpha_0,...,\alpha_3) \in \mathbb{C}^4\right\},$$
$$\Gamma (\mathbb{P}_{\Sigma}, \mathcal{O}(D_{\nabla_2})) =  \left\{\alpha_0 v_1 ... v_6 + \alpha_1 u_4^3 v_4^3 + \alpha_2 u_5^3 v_5^3 + \alpha_3 u_6^3 v_6^3 : (\alpha_0,...,\alpha_3) \in \mathbb{C}^4\right\}.$$
The description above is the Cox-homogeneous correspondent to the inside-torus description provided by the formula \eqref{eq:toric-inside-global-sections}. Since the fan $\Sigma$ is neither smooth nor simplicial, techniques provided by Cox description are not so useful for analyzing it. However, the description will be useful later, when we deal with the fan $\Pi$ which is a smooth subdivision of the fan $\Sigma$.

Consider two sections $s_1 \in \Gamma(O(\mathbb{P}_{K}, D_{\nabla_1})), s_2 \in \Gamma(O(\mathbb{P}_{K}, D_{\nabla_2}))$ defined as the zero set of the following expressions in the Cox ring:
\begin{equation}\label{eq:s-cox}
    s_1 = 3 \psi u_1 ... u_6 - u_1^3 v_1^3 - u_2^3 v_2^3 - u_3^3 v_3^3,   s_2 = 3 \psi v_1 ... v_6 - u_4^3 v_4^3 - u_5^3 v_5^3 - u_6^3 v_6^3.
\end{equation}
Equivalently, we can define them in terms of the description \eqref{eq:toric-inside-global-sections}:
\begin{equation}\label{eq:s-inside-torus}
s_1|_{\T_N} = 3 \psi - t_1 - t_2 - t_3, s_2|_{\T_N} = 3 \psi - t_4 - t_5 - t_6,
\end{equation}
where we assume that the two sections depend on the parameter $\psi \in \mathbb{P}^1$. Consider the family $\X_{\psi} = V(s_1, s_2)$, the common zero locus of the two aforementioned sections. The total space of the family lives in $\P_{\Sigma} \times \P^1$. For each $\psi \in \P^1$ the fiber of the family lives in $\P_{\Sigma}$.

The reasons for choosing this specific family are explained in \cite[Section 5]{Poch1}. Although the explanation is given for crepant desingularizations of the family $\X_{\psi}$, which we will consider shortly, the reasoning applies to $\X_{\psi}$ as well. 
For many technical reasons, it is more convenient to work with crepant desingularizations $\Y_{\psi}$ of the family $\X_{\psi}$ rather than with the family itself. The desingularizations are constructed using toric methods, as outlined in \cite{MPCP}. Specifically, we will work with the fan $\Pi$, a subdivision of the fan $\Sigma$. The associated toric variety $\P_{\Pi}$ is a smooth projective toric variety, and the induced morphism $\P_{\Pi} \to \P_{\Sigma}$ is a crepant resolution (cf. \cite[Chapter 4]{Poch1}).

\begin{theorem} 
    There is a smooth complete fan $\Pi$ with ray generators in the lattice $N$ which is a subdivision of the fan $\Sigma$. The fan $\Pi$ has $110$ rays, with the following set of ray generators:
    \[\Pi(1)^{\text{gen}} = \{u_{ijk} :  1 \leq i \leq j \leq k \leq 6, (i,j,k) \neq (1,2,3)\}\]
    \[\cup \{v_{ijk} :  1 \leq i \leq j \leq k \leq 6, (i,j,k) \neq (4,5,6)\},\]
    \[u_{ijk} = -e_1 - e_2 - e_3 + e_i + e_j + e_k,\]
    \[v_{ijk} = -e_4 - e_5 - e_6 + e_i + e_j + e_k.\]
\end{theorem} 

Consider a family $\Y_{\psi} = V(h_1,h_2) \subset \P_{\Pi}$, defined as the zero set of the following two Cox polynomials:
\begin{equation} \label{eq:h1}
    h_1  = 3 \psi \prod_{\substack{(i_1,i_2,i_3) \neq \\ (1,2,3)}} u_{i_1 i_2 i_3} - \sum_{t=1}^{3} \left(\prod_{\substack{(i_1,i_2,i_3) \neq \\(4,5,6)}} v_{i_1 i_2 i_3}^{\delta_{i_1}^t + \delta_{i_2}^t + \delta_{i_3}^t} \prod_{\substack{(i_1,i_2,i_3) \neq \\(1,2,3)}} u_{i_1 i_2 i_3}^{\delta_{i_1}^t + \delta_{i_2}^t + \delta_{i_3}^t}\right),
\end{equation}
\begin{equation} \label{eq:h2}
    h_2 = 3 \psi \prod_{\substack{(i_1,i_2,i_3) \neq \\ (4,5,6)}} v_{i_1 i_2 i_3} - \sum_{t=4}^{6} \left(\prod_{\substack{(i_1,i_2,i_3) \neq \\ (4,5,6)}} v_{i_1,i_2,i_3}^{\delta_{i_1}^t + \delta_{i_2}^t + \delta_{i_3}^t} \prod_{\substack{(i_1 i_2 i_3) \neq \\ (1,2,3)}} u_{i_1 i_2 i_3}^{\delta_{i_1}^t + \delta_{i_2}^t + \delta_{i_3}^t}\right).
\end{equation}

The Cox polynomials $h_1, h_2$ are the elements of the Cox ring $\C[r : r \in \Pi(1)^{gen}]$, they also can be regarded as elements of $\Gamma(\P_{\Pi},\O(D_1)), \Gamma(\P_{\Pi},\O(D_2))$ where $D_1$ and $D_2$ are the pullbacks of divisors $D_{\nabla_1}$ and $D_{\nabla_2}$ along the natural morphism $\pi$. They are computed explicitly in \cite[Proposition 5.1]{Poch1}. The family $\Y_{\psi}$ is the pullback of the family $\X_{\psi}$ along the natural morphism $\P_{\Pi} \to \P_{\Sigma}$. The geometry of the family is studied in \cite[Section 6]{Poch1} and summarized in \cite[Theorem 6.1]{Poch1}.


\subsection{Monodromy and Hodge Bundles} \label{sec-monodromy} Let $\Y \subset \P_{\Pi} \times \P^1$ be the total space of the family $\Y_{\psi}$, with the structural morphism $f : \Y \to \mathbb{P}^1$. Define $U = \mathbb{P}^1 \setminus (\{0, \infty\} \cup \mu_6)$, which is the domain of non-special parameters of the family $\Y_{\psi}$ and let $\Delta = \mu_6 \cup \{0, \infty\}$ . The restriction of the family 
$$f^{\times} = f|^{U}_{f^{-1}(U)}: f^{-1}(U) \to U$$ 
is a smooth proper map. We can consider the corresponding de Rham bundle $\H^p = R^p f_*^{\times} \mathbb{C} \otimes_{\mathbb C} \mathcal{O}_U$. We will make no distinction between locally free sheaves and vector bundles. For every point $\psi \in U$, the fiber $\H^3(\psi)$ can be naturally identified with $H^p(\Y_{\psi}, \C)$. Since the analytitfication of $\Y_{\psi}$ is a compact Kähler manifold, when $\psi \in U$, the vector space $H^p(\Y_{\psi}, \C)$ admits a Hodge filtration:
$$F^p H^p(\Y_{\psi}, \C) \subset ... \subset F^2 H^p(\Y_{\psi}, \C) \subset F^1 H^p (\Y_{\psi}, \C) \subset F^0 H^p(\Y_{\psi}, \C) = H^p(\Y_{\psi}, \C).$$
The filtration can be upgraded into a filtration of the bundle $\H^p$ via subbundles:
$$\F^p \H^p \subset \ldots  \subset \F^2 \H^p \subset \F^1 \H^p \subset \F^0 \H^p = \H^p.$$
The quotients $\H^{p,q} = \F^{p} \H^{p+q} / \F^{p+1} \H^{p+q}$ are called Hodge bundles. The de Rham bundle $\H^p$ is equipped with a flat Gauss-Manin connection $\nabla: \H^p \to \H^p \otimes_{\mathcal O_U} \Omega^1_U$, uniquely determined by the identity $\nabla(e \otimes h) = e \otimes dh$
for local (in analytic topology) sections $e \in (R^3 \pi_* \mathbb{C})(V), h \in \mathcal{O}_U^{an}(V)$, where $V \subset U^{an}$ is a simply-connected analytic subset of $U^{an}$. 

If $\xi \in \mathbb{P}^1$ is a point, we can consider a small analytic disk $D \subset \mathbb{P}^1$ centered at the point $\xi$ and a coordinate $t = \psi - \xi$. Choosing a base point $1 \in D$, we define the monodromy operator $T: \H^3(1) \to \H^3(1)$, which is given by the transportation of a vector via the connection $\nabla$ along a positively-oriented loop $\gamma \in \pi_1(D - \{\xi\}, 1)$. If $T$ is a unipotent operator, it can be uniquely represented as $e^N = T$, where $N$ is a nilpotent operator.

The limiting behavior of $\H^3(t)$ as $t \to 0$ can be studied using the limit mixed Hodge structure, which has been constructed in the work \cite{Schmid}. One can introduce a mixed Hodge structure on the (underlying rational) vector space $H^3_{\text{lim}} := \H^3(1)$  in such a way that the operator $\Gr_k N: \Gr_{3+k}^{W} H^3_{\text{lim}} \to \Gr_{3-k}^{W} H^3_{\text{lim}}$ is an isomorphism. In \cite[Theorem 8.5]{Poch1} and \cite[Example 2.15]{Ste}, the corresponding Hodge-Deligne numbers 
$$h^{p,q}_{\lim,a} = \dim_{\C} \Gr_F^p \Gr_{p+q}^W H^3_{\lim, a}$$ 
of the limit mixed Hodge structures near the special points were computed. These values are used in the computations in the subsection \ref{sec:zeros-subsection}.

The piece of the Hodge filtration $\F^3 \H^3$ is trivial over $U$. We can construct a global, everywhere nonzero section $\omega = \omega_0 = \omega_0(\psi) \in \Gamma(U, \F^3 \H^3)$ using the residues. Recall that $\Y_{\psi}$ is defined as the common zero set of two sections, $\Y_{\psi} = V(h_1,h_2) \subset \mathbb{P}_{\Pi}$. We denote $Y_{1,\psi} = V(h_1) \subset \P_{\Pi}$ and $Y_{2,\psi} = V(h_2) \subset \P_{\Pi}$, which are $4$-dimensional hypersurfaces in $\P_{\Pi}$. The toric variety $\mathbb{P}_{\Pi}$ has a canonical toric Euler form $\Omega$, as described in \cite[p369]{CLS}, the form is defined up to $\pm 1$, but if we choose an ordered basis of $N$ to be $e_1 - e_6, ..., e_5 - e_6$, the disambiguation vanishes.  One can define:
\begin{equation}\label{eq:omega-zero}
\omega_0 = \operatorname{res}_{Y_{1,\psi}} \operatorname{res}_{Y_{2,\psi}} \frac{(3\psi)^2 \Omega}{h_1 h_2},
\end{equation}
\begin{equation}\label{eq:omega-p}\omega_p = \nabla_{\psi \frac{d}{d\psi}} \omega_{p-1}, p = 1,2,3.
\end{equation}
The section $\omega_0$ is a global section of $\F^3 \H^3$ over $U$ which is nonzero at $\psi = \infty$, as written in \cite[Lemma 8.3]{Poch1}. Furthermore, the section $(3\psi)^{-2} \omega_0$ is nonzero at $\psi = 0$, as indicated in \cite[Lemma 8.2]{Poch1}. The section $\omega_p$ is an element of $\F^{3-p} H^3$ due to Griffiths transversality, and $\omega_0(\psi),...,\omega_3(\psi)$ form a basis of $\H^3(\psi)$ for every $\psi \in U$, as stated in \cite[Corollary 7.7]{Poch1}.

We denote by $[\omega_p(\psi)]$ the class of $\omega_p(\psi)$ in the quotient $F^{3-p}/F^{4-p} = H^{3-p, p}(\Y_{\psi})$. This we use to define 
\begin{equation}\label{eq:etadef}
    \eta_p(\psi) = [\omega_p(\psi)].
\end{equation}
By construction, $\eta_0 = \omega_0$. Since $F^4 H^3(\Y_{\psi},\C) = 0$, there is no need to pass to quotients. 

\section{Arithmetic Riemann-Roch}\label{sec-arithmetic-riemann-roch}
The arithmetic Riemann-Roch theorem relates arithmetic and analytic intersection data. In \cite{EFiMM3}, the analytic part was applied to compute the BCOV invariant of the mirror of Calabi-Yau hypersurfaces in projective space. Below we generalize this application to the situation considered in this paper.

In this section, we suppose that we are given a flat projective morphism $g: \Zv \to \P^1$ with smooth $n$-fold fibers outside a finite set $\Delta$ in $\P^1$. We write $g^\times$ for the restriction of $g$ to the complement of $\Delta.$ 

\subsection{A volume computation}
Consider the local system $R^k g^\times_\ast \C$ and the associated flat vector bundle $\H^k$. We suppose that the family $g$ is projective. In this setting, one can consider the $L^2$-norm  of the generating integral section of $\det H^k(\Zv_\psi, \Z)$. In \cite{EFiMM2} this is referred to as the volume of $H^k(\Zv_\psi, \Z).$ 

\begin{lemma}\label{lemma:deligneextension}
    Suppose that $\H^k$ be the vector bundle associated to local system $R^k g^{\times}_\ast \C$, and suppose that the local monodromies around $\Delta$ are unipotent. Then there exist a holomorphic trivializing section $\eta$ of $\det \H^k$. If the $L^2$-norm is induced by an integral K\"ahler class, then $\eta$ can be picked to have constant $L^2$-norm.
\end{lemma}

\begin{proof}
    Since the monodromies are unipotent, it follows that the determinant has trivial monodromies and hence is a constant local system. This implies the existence of $\eta$, which can be chosen, by parallel transport, as the integral generator of $\det H^k(\Zv_\psi, \Z).$ By \cite[Proposition 4.2]{EFiMM2}, this volume is constant. 
\end{proof} 

\subsection{An application of arithmetic Riemann--Roch}

In this subsection the main result is an application of the arithmetic Riemann--Roch theorem. In the below, $\chi(\Zv_{sm})$ refers to the topological Euler characteristic of any smooth fiber.

\begin{theorem}\label{th:ARR}
    Suppose we have projective flat morphism $g: \Zv \to \P^1 $ with smooth Calabi--Yau n-fold fibers outside a finite set $\Delta$ in $\P^1$. Suppose that 
    \begin{itemize}
        \item There exist meromorphic sections $\eta_0, \ldots, \eta_n$ that trivialize the corresponding Hodge bundle $\F^{n-p} \H^3 / \F^{n-p+1} \H^3$ over $\P^1 - \Delta$.
        \item The other flat vector bundles $\H^k, k \neq n,$ have unipotent local monodromies, and their Hodge structures are concentrated in terms of the form $H^{p,p}$.
    \end{itemize}
    Then, there is a rational function $\phi$ on $\mathbb{P}^1$, invertible outside $\Delta$, and a rational number $A$, such that 
    $$
        \log \tau_{BCOV}(z) =\frac{\chi(\Zv_{sm})}{12} \log\|\eta_0\|^2_{L^2} + (-1)^{n+1} \sum_{p=0}^{n} (n-p) \log \|\eta_p \|_{L^2}^2 +  A \log |\phi|.
    $$
\end{theorem}
\begin{proof}
    This is an application of the arithmetic Riemann--Roch theorem, as explained in \cite[Theorem 2.3]{EFiMM3}. There are additional terms involved, of the type $\|\eta_{p,q}\|^2_{L^2}$ for general $p,q$ being rational sections of the $\det H^{p,q}, p+q\neq n$. We claim these can be chosen so that they only contribute by constants. For this, we first notice that by assumption, only sections of the form $(p,p)$ contribute, which can be translated into a question about $\det \H^{p,p} = \det \H^{2p}.$
    
\end{proof}

Hence, to compute the BCOV-invariant in terms of the $L^2$-norms of the sections $\eta_p$, it suffices to determine $\phi$. Since rational functions on $\P^1$ are defined, up to a constant, by their zeros and poles, this reduces the problem to determining a set of discrete data. Specifically, determining $\phi$ requires identifying its zeros and poles at all but one point. In the current application, this determination is made for all points except $\infty$.

We incorporate the constant $A$ in $\phi$ and write 
$$
\phi = C \cdot \prod_{a_i \in \C} (z-a_i)^{n_i}.
$$
for some rational numbers $n_i.$ Then we find that
\begin{equation}\label{eq:nidef} 
    \log \BCOV(z) = \frac{\chi(\Zv_z)}{12} \log\|\eta_0\|^2_{L^2} - \sum  (n-p) \log \|\eta_p\|_{L^2}^2  = n_i  \log|z-a_i| + O(1)
\end{equation}
as $z \to a_i$. This reduces the problem to managing the growth $\BCOV$  and the various $\eta_p$ at $\Delta - \{\infty\}.$ 

\subsection{In dimension 3} In this section we specialize our considerations to the case $n=3$, and consider the problem of finding the growth of $\BCOV$ close to singular points. We cite the following version of a result from \cite[Theorem 7.6]{EFiMM2}: 

\begin{proposition}\label{th:kappa-unipotent}
    Let $\xi \in \Delta$, let $t$ be a local coordinate around $\xi$ centered at 0. Suppose that: 
    \begin{itemize}
        \item The local monodromy around $\xi$ of $R^k g^\times_\ast \C$ is unipotent, for all $k$,
        \item The model $\Zv \to \Delta$ is Kulikov, i.e. $g^\ast g_\ast K_{\Zv/\Delta} \simeq K_{\Zv/\Delta}.$
    \end{itemize} 
    Then we have the asymptotic expansion 
    $$
    \log \BCOV(t) = \kappa_{\xi} \log|t|^2 + o(\log|t|^2),
    $$
    where
    $$
    \kappa_{\xi} = -\frac{1}{6}\left( \chi(\Zv_{sm}) - \chi (\Zv_{\xi})\right) - \sum_{i=0}^{k-1} \chi (\mathcal O_{\widetilde{Z_i}})).
    $$
    Here we have denoted by 
    \begin{itemize}
        \item  $\chi(\Zv_{sm})$ and $\chi(\Zv_\xi)$ the topological Euler characteristic of a smooth and central fiber respectively.
        \item  $\{Z_0,...,Z_{k-1}\}$ the set of irreducible components of the central fiber $\Zv_\xi$, and $\widetilde{Z_i}$ is a desingularization of the irreducible component $Z_{i,red}$ and $\chi(\mathcal O_{\widetilde{Z_i}})$ the holomorphic Euler characteristic of $\mathcal O_{\widetilde{Z_i}}.$
    \end{itemize}

\end{proposition}

In the application, $\Delta = \{0,\infty\} \cup \mu_6$ as it was defined at Section \ref{sec-monodromy}. The fact that local monodromies are unipotent in our situation follows from \cite[Proposition 7.8, Lemma 8.1]{Poch1}, but will be recalled in Proposition \ref{prop-a11-is-zero}. The classes $\eta_p$ we will use are the ones defined in \eqref{eq:etadef}.

\section{Computation of the periods and formulation of the BCOV conjecture} \label{sec-periods}
In this section we consider the periods of our mirror family around infinity. We do this following the discussion of \cite[Section 5]{EFiMM3}. In particular, we construct an adapted basis of the weight filtration and in Theorem \ref{th:tbcov-a} we re-formulate the BCOV conjecture for the family mirror family of the intersection of two cubics. 

\subsection{A formal series}

Consider the formal series: 
\begin{equation}\label{eq:Rwt-def} 
R(w,t) = \sum I_{0,q}(t) w^{q}:= e^{w t} \sum_{d=0}^{\infty} e^{d t} \frac{\prod_{r=1}^{3d} (3w +r)^2}{\prod_{r=1}^{d} (w+r)^6}
\end{equation}
It is immediate to verify that $R(w,t)$, viewed as a formal power series in two variables, satisfies the following Picard--Fuchs type equation (compare \cite[(4.33)]{Popa}, and equation before (4.22) in her work):

\begin{equation}\label{eq:Popaequation}
    \left[\frac{d^4}{dt^4} -   3^6 e^t  \left( \frac{d}{dt}+\frac{1}{3}\right)^2 \left( \frac{d}{dt}+\frac{2}{3}\right)^2\right] R(w,t) = e^{w  t} w^4  .
\end{equation}
Hence for $q=0,1,2,3$ the series $I_{0,q}(t)$ are annihilated by
\begin{equation}\label{eq:Popaequation2}
    P = \frac{d^4}{dt^4} -  3^6 e^t \left( \frac{d}{dt}+\frac{1}{3}\right)^2 \left( \frac{d}{dt}+\frac{2}{3}\right)^2 .
\end{equation}

This is the  Picard--Fuchs equation of the mirror of the intersection of two cubics in $\mathbb{P}^5$, cf.  \cite[Example 6.4.2]{BvS}, \cite[Theorem 7.6]{Poch1}.
Setting 
\begin{equation} \label{eq:mirror-coordinate}
    e^{t} = z = \frac{1}{(3\psi)^6}, 
\end{equation}
the differential operator in  \eqref{eq:Popaequation2} takes the form:
\begin{displaymath}
    P = \left(z \frac{d}{dz}\right)^4 - 3^6 z \left(z \frac{d}{dz} + \frac{1}{3}\right)^2 \left(z \frac{d}{dz} + \frac{2}{3}\right)^2,
\end{displaymath}
or
\begin{displaymath}
    P = \left(\psi \frac{d}{d\psi}\right)^4 - \psi^{-6} \left(\psi \frac{d}{d\psi} - 2 \right)^2 \left(\psi \frac{d}{d\psi} - 4\right)^2,
\end{displaymath}
which is nothing but the Picard--Fuchs equation at $\infty$ of our mirror family, with respect to the holomorphic top form $\omega_0, $ as proven in \cite[Theorem 7.6]{Poch1}. This form of the operator shows that the differential equation is regular, and the formal solutions are automatically convergent in the parameter $t$ in a neighborhood of infinity as long as the real part of parameter $t$ stays in a bounded interval. This is discussed in \cite[p. 117, Theorem 3.1, Theorem 5.2]{CL}. 

If we write 
\begin{displaymath}
    \sum_{d=0}^{\infty} e^{d t} \frac{\prod_{r=1}^{3d} (3w +r)^2}{\prod_{r=1}^{d} (w+r)^6} = \sum \widetilde{I}_q(e^t) w^q
\end{displaymath}
the function $\widetilde{I}_q$ is hence holomorphic in $z=e^t.$ By the description of $R(w,t)$, we find that 
\begin{equation}\label{eq:holomorphicdecomposition}
    I_{0,q}(t) = \widetilde{I}_q(z) +  t \widetilde{I}_{q-1}(z)+\frac{t^2}{2}\widetilde{I}_{q-2}(z) +\frac{t^3}{3!}\widetilde{I}_{q-3}(z).
\end{equation}
By construction, we have 
\begin{displaymath}
    I_{0,0}(t) = \widetilde{I}_0(z)=  \sum_{d \geq 0} e^{d t} \frac{((3d)!)^2}{(d!)^6}
\end{displaymath}
and one computes that 
\begin{displaymath}
    I_{0,1}(t) = \widetilde{I}_1(e^t) + t \widetilde{I}_0(e^t) =   \sum_{d \geq 1} e^{d t} \left(\frac{((3d)!)^2}{(d!)^6} \sum_{r=d+1}^{3d} \frac{6}{r}\right)+t \cdot I_{0,0}(t). 
\end{displaymath}
The leftmost term is denoted by $J(t)$ in \cite{Popa}. Then 
\begin{equation}\label{eq-mirror-coordinate-2}
    t \mapsto Q (t)= e^t e^J 
 = \exp\left(\frac{I_{0,1}(t)}{I_{0,0}(t)}\right)
\end{equation}
is the mirror variable as defined in \cite[(1.6), p. 1151]{Popa}. Since $Q$ is invariant under $t \mapsto t + 2\pi i$, this is a function of $z = e^t$. It is also the same mirror variable as the one considered in \cite[\textsection{2}]{Mor}, with the fixing of an implicit constant so that $I_{0,1}(t) - t I_{0,0}(t)$ has no constant term in $z.$

\subsection{Computation of periods}  We will use the notation $\H_3 = (\H^3)^{\vee}$ to indicate flat vector bundle dual to de Rham bundle $\H^3$. 

For $\omega_0$ the holomorphic top form in \eqref{eq:omega-zero}, the map  
$$\gamma \mapsto \int_\gamma \omega_0$$
induces an isomorphism between the local system $R^3 f^{\times}_* \C$ and the solutions of the Picard--Fuchs equation, for dimension reasons. We denote by $\gamma_j$ the multivalued flat homology 3-cycles such that 
\begin{displaymath}
    I_{0,q}(t) = \int_{\gamma_q(t)} \omega_0. 
\end{displaymath}
The function $I_{0,q}$ is defined via the formal series \eqref{eq:Rwt-def}. We define 
$$\widetilde{\gamma}_q = \exp\left(-t N \right) \gamma_q(t)$$ 
the corresponding single-valued (non-flat) sections of $\H_3$. Here $t$ is a coordinate on a universal cover of a small analytic punctured disk centered at infinity, defined via the formula \eqref{eq:mirror-coordinate}. Then we can reciprocally write 
\begin{displaymath} \gamma_q(t) = \widetilde{\gamma}_q(t) + t N \widetilde{\gamma}_q(t) + \frac{t^2}{2!} N^2 \widetilde{\gamma}_q(t) + \frac{t^3}{3!} N^3 \widetilde{\gamma}_q(t),
\end{displaymath}
so that
\begin{equation}\label{eq:singlevsmultivalued1} \int_{\gamma_j(t)} \omega_0 = \int_{\widetilde{\gamma}_{q}(t)}\omega_0  + t \int_{N \widetilde{\gamma}_{q}(t)}\omega_0 + \frac{t^2}{2} \int_{N^2 \widetilde{\gamma}_{q}(t)}\omega_0+ \frac{t^3}{6} \int_{N^3 \widetilde{\gamma}_{q}(t)}\omega_0,
\end{equation}
$$\widetilde{I}_q = \int_{\widetilde{\gamma}_q} \omega_0.$$
Moreover, one verifies from definition that the section $N \widetilde{\gamma}_j$ is a single-valued (non-flat) section of $\H_3$, so the integrals appearing in \eqref{eq:singlevsmultivalued1} are single-valued functions which are holomorphic in $z = e^t.$ Inspection of the solutions in \eqref{eq:holomorphicdecomposition} show us that
\begin{displaymath}
    \int_{N^k \widetilde{\gamma}_{q}(t)} \omega_0  = \int_{\widetilde{\gamma}_{q-k}(t)} \omega_0.
\end{displaymath} It follows from the description in \eqref{eq:singlevsmultivalued1} that that 
\begin{displaymath}
    \int_{N \gamma_{q+1}(t)} \omega_0 = \int_{\gamma_q} \omega_0,
\end{displaymath}
and hence from the correspondence between cycles and solutions of the Picard--Fuchs equation that $N \gamma_{q} = \gamma_{q-1}$. We summarize these observations in the following proposition:

\begin{proposition}
    The functions $I_{0,q}$ for $q = 0,1,2,3$ are linearly independent solutions to the Picard--Fuchs equation \eqref{eq:Popaequation2}. There are multivalued cycles $\gamma_q$ such that
    \begin{itemize}
        \item $I_{0,q}(t) =\int_{\gamma_q(t)} \omega_0,$ 
        \item $N\gamma_{q} = \gamma_{q-1}$,
        \item     $\gamma_0$ is a  single-valued flat cycle whose whose tube in $H_5( \P_{\Pi} - (Y_{1,t} \cup Y_{2,t}))$ is $(S^{1})^5 \subseteq (\mathbb{C}^\times)^5 \cong \T_N \subseteq \P_{\Pi}.$ See \cite[Proposition 7.2]{Poch1} for a discussion about the tube map.
    \end{itemize}  
    
\end{proposition}

\subsection{An adapted basis and other periods} \label{subsec:adaptedbasis}

Having defined $I_{0,q}$ in the previous subsection, we further define inductively on $p$,
\begin{equation}\label{eq:Ipq}
    I_{p,q} = \frac{d}{dt} \frac{I_{p-1,q}}{I_{p-1,p-1}}.
\end{equation}
With this setup, $I_{p,p}$ is the same as the $I_p$ defined in \cite[(1.4)]{Popa}. As considered in \cite[Proposition 5.5]{EFiMM3}, we have the following construction, which includes the inductive construction of an adapted basis to the weight filtration determined by the $\gamma_q$: 

\begin{proposition}\label{prop:Ipq}
    The following statements hold: 
    
    \begin{enumerate}
        \item There are sections, ${\omega}_p$ of $\F^{3-p} \H^3$,  for $p=0,1,2,3$, with ${\omega}_0$ as in Section \ref{sec-monodromy}, and then inductively defined by
        \begin{displaymath}
            {\omega}_p = \frac{d}{dt} \left(\frac{{\omega}_{p-1}}{I_{p-1,p-1}}\right)
        \end{displaymath}
        such that they satisfy  
        \begin{displaymath} I_{p,q}(t) = \int_{\gamma_q(t) }{\omega}_p.
        \end{displaymath}
        \item \label{prop:Ippproperties}The series $I_{p,p}$ are holomorphic in $z$ and converge in a neighborhood of $\infty$ and $I_{p,p}(\infty) = 1.$ \done{ do we need this? Also, suddenly not sure if this is the right definition $I_{4,4}$, there is a slight deviation in the literature from these definitions} They also satisfy that 
        $$I_{4-p,4-p}  = I_{p,p}$$ and 
        $$I_{0,0} I_{1,1} I_{2,2} I_{3,3} I_{4,4}= \left(1-3^6 e^t \right)^{-1}.$$
        \item The sections satisfy that 
        \begin{displaymath}
            I_{p, q}(t) = \int_{\gamma_q(t)} {\omega}_p = 0, \hbox{ for } q < p.  
        \end{displaymath}
        In particular $$\widetilde{\omega}_p = \frac{\omega_p}{I_{p,p}}$$
        is an adapted basis to the weight filtration. 

        \item \label{item4:Ipq} For $p = 0, 1, 2, 3$, the projection of $\tilde{\omega}_p$ onto the Hodge bundle $\H^{p,3-p}$  satisfy 
        \begin{displaymath}
            \widetilde{\eta}_p := \left[\widetilde{\omega}_p\right] = \frac{(-1)^p}{6^p} \cdot \frac{\eta_p}{\prod_{k=0}^{p}  I_{k,k}},
        \end{displaymath}
        where $\eta_p$ is defined in \eqref{eq:etadef}.
        \done{what is the definition of $\eta_k$? Depending on the variable we differentiate with respect to we get a different normalizing factor $C$. The above definition is with respect to $\frac{d}{dt} = -\frac{1}{6} \psi \frac{d}{d \psi}$ and supposing that $\eta_p$ is obtained by differentiating by $\left(\psi \frac{d}{d\psi}\right)^p$}
        \end{enumerate}
\end{proposition}
\begin{proof}
The symmetry in the third point follows from \cite[Proposition 4.4]{Popa}, generalizing the main results of \cite{ZagierZinger}.

The proof of the other points follows the same outline as in \cite[Proposition 5.5]{EFiMM3} and we refer the reader to this for further details. 
\end{proof}

\subsection{Comparison with Popa}
We recall the computation of \cite[Corollary 2]{Popa}, which applies to the generating series of genus one Gromov--Witten invariants on a $X_{3,3}$, a smooth complete intersection of two cubics in $\mathbb{P}^5$. Denote by $H$ a generic hyperplane in $\mathbb{P}^5$, restricted to $X_{3,3}.$ We will consider the series: 
\begin{equation}\label{definition:F1A}
    F_{1,A}(Q) = N_1^0(X_{3,3}) \log Q + \sum_{d=1}^{\infty} N_1^d(X_{3,3}) Q^d,
\end{equation}
$$ N_1^0(X_{3,3}) := \frac{-1}{24}c_2(\Omega_{X_{3,3}}) \cap [H],$$
where we implicitly have taken degrees of the cycle. Here, $N_1^d(X_{3,3})$ are genus one Gromov-Witten invariants.

\begin{proposition}\label{prop:Popacomputations}
The following identities hold:

\begin{enumerate}
    \item  $N_{1,0}(X_{3,3}) =  \frac{-9}{4}.$
    \item \label{item2:popacomputations} The series defined in \eqref{definition:F1A} are given by: 
    $$
    F_{1,A}(Q)  =N_1^0(X_{3,3}) t - \frac{7}{12}\log(1-3^6 e^t) + \frac{\chi(X_{3,3})}{24}\log I_{0,0}(t) - \frac{1}{2} \left(\sum_{p=0}^4 { \binom{4-p}{2}} \log I_{p,p}  \right) = 
    $$
    $$
    \frac{-9}{4}\log t - \frac{7}{12}\log(1-3^6 e^t) - 6 \log I_{0,0}(t) - \frac{1}{2} \left(6 \log I_{0,0} + 3 \log I_{1,1} + \log I_{2,2} \right).
    $$
\end{enumerate}

\end{proposition}

\begin{proof}
    The first point is a standard computation by resolving $\Omega_{X_{3,3}}$ by the cotangent exact sequence, reducing to known computations on projective space. One finds that $c_2(\Omega_{X_{3,3}}) = 6 \cdot [H]^2$ and one concludes since $[H]^3=9.$

    To prove the second point, we start with a special case of  \cite[Corollary 2]{Popa}, taking into account \eqref{eq-mirror-coordinate-2}, which gives
    \begin{displaymath}
F_{1,A}(Q)  = \frac{-9}{4}\log t -\frac{1}{12}\log(1-3^6 e^t) -8 \log I_{0,0}(t) - \frac{1}{2}\log I_{1,1} 
    \end{displaymath}

By Proposition \ref{prop:Ipq} \eqref{prop:Ippproperties}, we have that 
\begin{displaymath}
    \log I_{0,0}+\log I_{1,1}+\frac{1}{2}\log I_{2,2}= \frac{-1}{2} \log (1-3^6 e^t)
\end{displaymath}
which implies that 
\begin{displaymath}
    2 \log I_{0,0} + \frac{1}{2}\log I_{1,1} = \frac{1}{2}\log(1-3^6 e^t)+ 3 \log I_{0,0}+\frac{3}{2} \log I_{1,1}+\frac{1}{2} \log I_{2,2}
\end{displaymath}
and hence that $F_{1,A}$ is given by
\begin{equation}\label{eq:F1Aexplicit}
\frac{-9}{4}\log t - \frac{7}{12}\log(1-3^6 e^t) - 6 \log I_{0,0}(t) - \frac{1}{2} \left(6 \log I_{0,0} + 3 \log I_{1,1} + \log I_{2,2} \right).
\end{equation}
Since $\chi(X_{3,3})= -144$ so that $\chi(X_{3,3})/24 = -6,$ the expression \eqref{eq:F1Aexplicit} equals the cla\-imed expression. 
\end{proof}

\subsection{Reformulating the conjecture} 
As a consequence of the above computations, we can reformulate the main conjecture in \eqref{EFiMM:BCOV-conjecture} as follows, where $\eta_p$ are the classes defined in \eqref{eq:etadef}:
\begin{theorem}\label{th:tbcov-a}
The conjecture is equivalent to the formula, up to constant,  
    $$\BCOV(\Y_\psi) = \left|\psi^{-68}(\psi^6-1)^{7/3} \right| \cdot \|\eta_0\|_{L^2}^{\chi/6} \cdot \|\eta_0\|_{L^2}^6 \cdot \|\eta_1\|_{L^2}^4 \cdot \|\eta_2 \|_{L^2}^2.$$ 
Here $\chi  = 144$ denotes the topological Euler characteristic of a smooth fiber $\Y_\psi, \psi \in U$. 
\end{theorem}
\begin{proof}
    The BCOV conjecture as stated in \eqref{EFiMM:BCOV-conjecture} claims that we should have 
    $$\BCOV(\Y_\psi) =C | \exp(-F_{1,A})|^4 \cdot \|\widetilde{\eta}_0\|_{L^2}^{\chi/6}   \cdot \|\widetilde{\eta}_0\|_{L^2}^6 \cdot \|\widetilde{\eta}_1\|_{L^2}^4\cdot \|\widetilde{\eta}_2 \|_{L^2}^2.$$
    By Proposition \ref{prop:Ipq} \eqref{item4:Ipq}, we have that, up to constant,
    \begin{displaymath}
        I_{0,0}^{\chi/6} I_{0,0}^{6+4+2} I_{1,1}^{4+2} I_{2,2}^2 \|\widetilde{\eta}_0\|_{L^2}^{\chi/6}   \cdot \|\widetilde{\eta}_0\|_{L^2}^6 \cdot \|\widetilde{\eta}_1\|_{L^2}^4\cdot \|\widetilde{\eta}_2 \|_{L^2}^2 = \|{\eta}_0\|_{L^2}^{\chi/6}   \cdot \|{\eta}_0\|_{L^2}^6 \cdot \|{\eta}_1\|_{L^2}^4\cdot \|{\eta}_2 \|_{L^2}^2
    \end{displaymath}
    and inserting the expression for $F_{1,A}$ in Proposition \ref{prop:Popacomputations} \eqref{item2:popacomputations} we find, up to constant since $3^6 e^t = \psi^{-6}$,
    $$ | \exp(-F_{1,A})|^4 \cdot \|\widetilde{\eta}_0\|_{L^2}^{\chi/6}   \cdot \|\widetilde{\eta}_0\|_{L^2}^6 \cdot \|\widetilde{\eta}_1\|_{L^2}^4\cdot \|\widetilde{\eta}_2 \|_{L^2}^2 =$$  
    $$\left|(\psi^{-6})^{9/4}(1-\psi^{-6})^{7/12}\right|^4 \|{\eta}_0\|_{L^2}^{\chi/6}   \cdot \|{\eta}_0\|_{L^2}^6 \cdot \|{\eta}_1\|_{L^2}^4\cdot \|{\eta}_2 \|_{L^2}^2.
    $$
Since   
    $$\left|(\psi^{-6})^{9/4}(1-\psi^{-6})^{7/12}\right|^4   =   \left|\psi^{-68}(\psi^6-1)^{7/3} \right|$$
the genus one BCOV conjecture is equivalent to the formula in the statement. 
\end{proof}
Combining the just proven Theorem \ref{th:tbcov-a} with Theorem \ref{th:ARR} we find that the following version of our BCOV conjecture:
\begin{corollary}\label{cor:formulationBCOVgenusone}
    The BCOV conjecture is equivalent to proving that $$|\phi| =  \left|\psi^{-68}(\psi^6-1)^{7/3} \right|.$$
\end{corollary}
The rest of the paper is dedicated to proving this formula. We notice that the function $\phi$ has zeros and poles in the points corresponding to the singular fibers of $\Y \to \mathbb{P}^1.$ We follow the strategy of trying to determine these exponents by looking at the asymptotics of the quantities involved in Theorem \ref{th:tbcov-a} in the sense of  \eqref{eq:nidef}. This is the content of the next two sections. 

\section{$L^2$ growth of the constructed sections}\label{sec-order-of-zeros}

From Theorem \ref{th:ARR}, it follows that to understand how $|\phi|$ behaves, one needs to understand how $||\eta_p||_{L^2}$ behaves near the singular points of $\Delta$. In this section we will compute the order of zeros of the sections $\eta_p$ at various points of the moduli, which by  \eqref{eq:L2growthzeros} below allow us to control the behaviour of its $L^2$-norm.

\subsection{Schmid's asymptotics} \label{ch:schmid-subsec}
To understand better the behaviour of the $L^2$-norm of these sections on a unit disc, with coordinate $z$ with a singularity at $z=0$, we rely on Schmid's work on Hodge norm estimates \cite[Theorem 6.6]{Schmid}. All the local systems/flat vector bundles in this section have unipotent monodromies, and we continue with this assumption.

It admits several reformulations or improvements, and the one we need appears in Theorem 4.4 and Remark 4.5 of \cite{EFiMM2}. To state this, consider the Deligne extension  $\widetilde{\H}^3$  of $\H^3,$ and also the Deligne extension $\widetilde{\F}^p$ 
 of the Hodge filtration $\F^p$. The quotients $ \widetilde{\H}^{p,q}=
\widetilde{\F}^p/\widetilde{\F}^{p+1}$ are in general locally free, and in our setting they are even one-dimensional, which we suppose. The fiber at zero of $\widetilde{\H}^3$ naturally identifies with the corresponding limit mixed Hodge structure $H^{3}_{\lim}$, and the fiber of $\widetilde{\H}^{p,q}$ identifies with  $H^{p,q}_{\lim}$.

The reference above implies that if the monodromy is unipotent and $\omega$ is a holomorphic trivializing section $\widetilde{\H}^{p,q}$, then the growth of the $L^2$-norm is at worst  logarithmical. If $\ell$ is an integer such that $z^{-\ell} \omega$ is non-vanishing on the Deligne extension, it means that 
\begin{equation}\label{eq:L2growthzeros}
   \log \|\omega \|_{L^2} = \ell \cdot \log|z| + o(\log|z|),
\end{equation}
so that the dominant behaviour of the $L^2$-norm is governed by the vanishing order of $\omega$. This will be determined in the following sections. 

\subsection{Yukawa coupling} In order to compute the orders of zeros of our sections, following the strategy in \cite{EFiMM3}, we first compute the Yukawa coupling. Consider the non-normalized Yukawa coupling 
$$ Y(\psi) = \int_{\Y_{\psi}} \omega_0 \wedge \nabla_{\psi \frac{d}{d \psi }}^3 \omega_0.$$
This depends on $\omega_0$ and the coordinate $\psi$.

\begin{lemma}\label{yukawalemma}
    The non-normalized Yukawa couplings are equal to 
    $$Y(\psi) = C \frac{\psi^6}{\psi^6 - 1},$$
    for any $\psi \in \Delta$, with non-zero constant $C$.
\end{lemma}
\begin{proof}
    We will do the computation in the coordinate $z = (3\psi)^{-6}$. The computation of the Yukawa coupling is standard, and follows as in \cite[Section 5.6]{CK} from a knowledge about the Picard-Fuchs equations. According to \cite[Section 5.6.1]{CK}, we find that 
    \[
    D Y =  -\frac{1}{2} \left( - \frac{2 \cdot 3^6 z}{1 - 3^6 z} \right) Y,
    \]
    where $D = z \frac{d}{dz}$, so that 
    \[
    Y = C\exp\left(\int \frac{ 3^6 }{1 - 3^6 z} d z\right) = C \frac{1}{1 - 3^6 z}
    \]
    \[
    Y = C \frac{1}{1 - 3^6 (3 \psi)^{-6}} = C \frac{ \psi^6 }{\psi^6 - 1} 
    \]
    for some $C$. In our case, $z = 0$ is the MUM point and \cite[Corollary 4.5.6]{BvS} states that the constant $C$, namely $Y(0)$, is non-zero in our setting.  
\end{proof}

\subsection{Order of zeros} \label{sec:zeros-subsection}
In this subsection, we compute the order of the zeros of the sections $\eta_0,...,\eta_3$, defined in \eqref{eq:etadef}, near points of $\Delta$. Here the sections are viewed as sections of the Deligne extensions of $\H^{p,q}$. The zeros of the sections near ordinary double points can be computed using the same ideas as in \cite[Theorem 4.8]{EFiMM3} :
\begin{proposition} 
    Let  $\xi \in \mu_6$. Then:
    $$\ord_{\psi=\xi}(\eta_0) = \ord_{\psi=\xi}(\eta_1) = 0,$$
    $$\ord_{\psi=\xi}(\eta_2) = \ord_{\psi=\xi}(\eta_3) = 1.$$
\end{proposition}
We need an additional argument for $\psi = 0$. Since the fiber $\Y_0$ is a singular fiber which is a simple normal crossing divisor \cite[Theorem 6.1(8)]{Poch1}.  
\begin{proposition}
    We have
    $$\ord_{\psi = 0} (\eta_0) = \ord_{\psi = 0} (\eta_1) = 2,$$
    $$\ord_{\psi = 0} (\eta_2) = \ord_{\psi = 0} (\eta_3) = 4.$$
\end{proposition}

\begin{proof}
    
 From Lemma \ref{yukawalemma} we have
    \begin{equation}\label{eq:yukawaaround0}
        \int_{\Y_{\psi}} (3\psi)^{-2} \omega_0 \wedge \nabla^3_{\psi \frac{d}{d\psi}}  (3\psi)^{-2} \omega_0 = C \frac{\psi^2}{\psi^6 - 1}, \psi \in U,
    \end{equation}
    for some $C \neq 0$ that is independent on $\psi$. Consider the sequence
    $$\widetilde{\F}^3/\widetilde{\F}^4 \stackrel{\text{KS}_3}{\longrightarrow} \widetilde{\F}^2/\widetilde{\F}^3 \stackrel{\text{KS}_2}{\longrightarrow} \widetilde{\F}^1/\widetilde{\F}^2 \stackrel{\text{KS}_1}{\longrightarrow} \widetilde{\F}^0/\widetilde{\F}^1$$ 
    of rank one vector bundles. The maps are Kodaira-Spencer maps, which means that they act by the rule
    $$\omega \mapsto \left[\nabla_{\psi \frac{d}{d \psi}} \omega \right] \mod \widetilde{\F}^{i+1}.$$
    As such they are a priori meromorphic maps, and we are interested in their order of vanishing, i.e. to what extent they are not isomorphisms. As the Deligne extension is determined by the cohomology of the relative logarithmic de Rham complex (cf. \cite{Ste}), and $\psi \frac{d}{d\psi}$ is a logarithmic vector, the above maps are in fact holomorphic. Hence we at least know that $\ord_0(KS_i) \geq 0.$ 
    Since we know that $\psi^{-2} \omega_0|_0 \neq 0$ from \cite[Lemma 8.2]{Poch1} and we have set $\omega_0 = \eta_0$, $\psi^{-2} \eta_0$ trivializes leftmost bundle. 
    
    This means that the order of vanishing of the composition of all the Kodaira-Spencer maps correspond to the order of vanishing of the section of $\widetilde{\F}^0/\widetilde{\F}^1$ obtained by applying all the Kodaira--Spencer maps to $\psi^{-2} \eta_0$. 
    Since the composition of all of these corresponds to the Yukawa coupling, with respect to the section $\psi^{-2} \eta_0$, which by \eqref{eq:yukawaaround0} vanishes to order 2 at the origin, we find that 
    $$\ord_{0}(KS_3) + \ord_0(KS_2) + \ord_0(KS_1) = 2. $$
     From  \cite[Lemma 8.1]{Poch1} we know that the local system $R^3 f_* \C $ is unipotent. This means that the Kodaira-Spencer map $KS_i$ restricted onto the zero fiber is $\Gr_i N^i$. $\Gr_i N^i$ is an isomorphism for $i=1,3$, which means that $\ord_0(KS_2) = 2$. Since we know that $\psi^{-2} \eta_0|_0 \neq 0$ 
     it follows that
    $$\ord_0(\eta_0) = 2, \ord_0 (\eta_1) = 2, \ord_0 (\eta_2) = 4, \ord_0(\eta_3) = 4.$$
\end{proof}
Finally we record the following, which is technically not needed. The proof is similar to the above and is hence omitted. 
\begin{proposition}
    For $p= 0,1,2,3$, we have
    $$\ord_{\psi = \infty} (\eta_p) = 0.$$
\end{proposition}
Those propositions, combined with the Hodge norms estimates in \textsc{\ref{ch:schmid-subsec}}, can help us to compute the asymptotic of $L^2$-norms of sections $\eta_i$.

\section{Asymptotics of the BCOV invariant} \label{sec-asymptotics-of-BCOV}
The goal of this section is to compute the asymptotical behaviour of BCOV-invariant near 8 special points, namely, near the point of $\Delta = \mu_6 \cup \{0,\infty\}$.  Around each point $\psi \in \{0,\infty\} \cup \mu_6$ we want to compute a rational number $\kappa_{\psi} \in \Q$ such that
$$\log \tau_{BCOV} (\Y_{t}) = \kappa_{\psi} \log |t|^2 + o(\log |t|^2), \text{ when } t \to 0$$
where $t$ is a local coordinate centered at a special point $a$. For $a \in \mu_6$, we can apply \cite[Theorem 7.3]{EFiMM3}, this gives us, $\kappa_{a} = \frac{1}{6}$ when $a \in \mu_6$. We also need to know that the monodromy of $\H^2$ is unipotent in order to be sure that the assumptions of Proposition \ref{th:kappa-unipotent} are satisfied.

\begin{proposition} \label{prop-a11-is-zero} 
    The local monodromy of de Rham bundles $\H^i$ is unipotent near any of the singular points.
\end{proposition} 
\begin{proof}
    We consider the various bundles for different $i$. \\
    $i=0,6$: The monodromy for $\H^0$ is trivial, and so is the one for $\H^6$ by duality. \\
    $i=2,4$: Likewise, if we can establish the case of $\H^2$, the case for $\H^4$ follows. For $\H^2$, we notice that around the ordinary double points, the monodromy is actually trivial. This also implies that the monodromy around $\infty$ for this bundle is expressed in terms of that around $0$. Finally, the fiber $\Y_0$ is a simple normal crossing divisor in a smooth total space $\Y$, as stated in \cite[Theorem 6.1(8)]{Poch1}, so hence this monodromy is unipotent. \\ 
    $i=3$: For $\H^3$ the discussion around $0$ is the same as for $\H^2$. The case of $\infty$ can be found in \cite[\textsection{7}]{Poch1}, and around the double points one uses that in odd dimensions the monodromy unipotent, by the Picard--Lefschetz formula. \\
    $i=1,5$: The other non-treated bundles, $\H^1$ and $\H^5$ are actually zero. 
\end{proof}

\subsection{Geometry around 0} The second thing that we want to do is to compute $\kappa_0$. We need some of the invariants computed, namely, we want to use Proposition \ref{th:kappa-unipotent}, so we need to know the following Euler characteristics: 
\begin{itemize}
    \item Topological Euler characteristics $\chi(\Y_0)$ of a special fiber;
    \item Topological Euler characteristics $\chi(\Y_{sm})$ where $\Y_{sm}$ means smooth fiber, fiber over a point $\psi \in U$;
    \item Holomorphic Euler characteristics $\chi(\O_{W_0}), \chi(\O_{W_1}), \chi(\O_{W_2})$, here $W_0 \cup W_1 \cup W_2 = \Y_0$ is decomposition in the union of irreducible components, see \cite[Theorem 6.1]{Poch1}.
\end{itemize}

Here, by topological Euler characteristic, we mean Euler characteristic of the analytifications $\Y_{sm}^{an},\Y_0^{an}$ of the varieties $\Y_{sm}, \Y_0$, we will drop "an" in the notation $\chi(\Y_0)$ for clarity. We will rely heavily on computer algebra systems for these computations. At this point, we need actual geometric information about the family $\Y_{\psi}$.  

For each cone $\sigma \in \Pi$ let 
$$\T_{\sigma} = \spec \C[\sigma^{\vee} \cap M] \cap \bigcap_{r \in \sigma(1)^{gen}} D_r$$ 
be the corresponding toric stratum. Here, $D_r \subset \P_{\Pi}$ is the simple toric-invariant divisor associated with the ray generator $r$. In particular, for a $k$-dimensional cone we have $\T_{\sigma} \cong (\C^{\times})^{5-k}$. So there is the decomposition
$$\P_{\Pi} = \bigcup_{\sigma \in \Pi} \T_{\sigma}$$
of the toric variety $\P_{\Pi}$ in disjoint union of semialgebraic sets. The decomposition induce decomposition of $\Y_{\psi} \subset \P_{\Pi}$ onto semialgebraic sets. This decomposition will be used later in the proof of Proposition \ref{th:euler-characteristics-of-Y0}.

The proof of Proposition \ref{th:euler-characteristics-of-Y0} uses BKK-formula (cf. \cite[Theorem 3]{BernsteinHovanski}) for each toric stratum $\T_{\sigma} \cap \Y_{0}$, however, there is a technical difficulty: the BKK-formula holds only for open set of polynomials, which means that it can, in principle, does not hold for $\Y_0 \cap \T_{\sigma}$. The next lemma shows that we still can use the formula even in such a situation:

\begin{lemma} \label{th:BKK-lemma}
    Let $h_1|_{\psi=0, \T_{\sigma}}$ and $h_2|_{\psi=0, \T_{\sigma}}$ be the Laurent polynomials, obtained by restricting the Cox polynomials \eqref{eq:h1} and \eqref{eq:h2} to $\T_{\sigma}$ and setting $\psi = 0$. Let $\Delta_1, \Delta_2$ be the Newton polytopes of the corresponding Cox polynomials. The next formula holds
    $$\chi(\Y_{0} \cap \T_{\sigma}) = (-1)^{5 - \dim \sigma} (5 - \dim \sigma)! \sum_{1 \leq i_1 \leq i_2 \leq i_3 \leq 2} \operatorname{Vol}(\Delta_{1,\sigma}, \Delta_{2,\sigma}, \Delta_{i_1, \sigma}, \Delta_{i_2, \sigma}, \Delta_{i_3, \sigma}),$$
    where $\chi$ stays for Euler characteristics. 
\end{lemma}
\begin{proof}
    The variety $\Y_0$ defined as the common zero set of the sections $h_1,h_2$ (given in formulas \eqref{eq:h1},\eqref{eq:h2}, respectively), when we set $\psi = 0$. Until the end of the proof, we will assume that $\psi = 0$, so we will use $h_1$ and $h_2$ to indicate $h_1|_{\psi = 0}$ and $h_2|_{\psi = 0}$. We want to use the BKK-formula, as presented in \cite[Theorem 3]{BernsteinHovanski}. The formula holds for a Zariski open subset of pairs of Laurent polynomials with a fixed Newton polytope.

    
    We can consider the family $$\widetilde{\Y}_{a_1,a_2,a_3,a_4,a_5,a_6} = V(a_1 b_1 + a_2 b_2 + a_3 b_3, a_4 b_4 + a_5 b_5 + a_6 b_6),$$
    $$b_j = \prod_{\substack{(i_1,i_2,i_3) \neq \\ (1,2,3)}} u_{i_1 i_2 i_3}^{\delta_{i_1}^{j} + \delta_{i_2}^{j} + \delta_{i_3}^{j}}  \prod_{\substack{(i_1,i_2,i_3) \neq \\ (4,5,6)}} v_{i_1 i_2 i_3}^{\delta_{i_1}^{j} + \delta_{i_2}^{j} + \delta_{i_3}^{j}}.$$
    Each fiber of the family is a subset of $\P_{\Pi}$. If all the parameters $a_i \neq 0$, then the fiber $\widetilde{\Y}_{a_1,...,a_6}$ is isomorphic to $\Y_0$. Therefore, the Euler characteristic is the same for each fiber, as long as $a_i \neq 0$. 
    
    Let $\Gamma(\Delta_1)$ and $\Gamma(\Delta_2)$ be the vector spaces generated by the integer points of $\Delta_1$ and $\Delta_2$ respectively. Suppose we know that for each $\sigma \in \Pi$ the canonical map 
    $$p_{\sigma} : \C\{b_1,b_2,b_3\} \times \C\{b_4,b_5,b_6\} \to \Gamma(\Delta_{1, \sigma}) \times \Gamma (\Delta_{2, \sigma})$$
    is surjective. The map takes the corresponding pair of Cox polynomials $a_1 b_1 + a_2 b_2 + a_3 b_3, a_4 b_4 + a_5 b_5 + a_6 b_6$ and restrict them on $\T_{\sigma} \subset \P_{\Pi}$. If such a map is surjection, we can find a set of nonzero parameters $(\tilde{a}_1,...,\tilde{a}_6) \in \C\{b_1,b_2,b_3\} \times \C\{b_4,b_5,b_6\}$ for which the BKK-formula holds after restriction on $\T_{\sigma}$. But this means that it also holds for $a_1 = 1, \ldots, a_6 = 1$, since Euler characteristics and Newton polytopes remain the same regardless of the choice of parameters, as long as they are nonzero.
    
    Checking surjectivity of the above-mentioned morphism is the same thing as checking that polytopes $\Delta_{1,\sigma}, \Delta_{2,\sigma}$ have no other integer points, except those which comes from the monomials of $h_1|_{\T_{\sigma}}, h_2|_{\T_{\sigma}}$. We have verified this with computer algebra systems, see Appendix \ref{ch:appendix-holomorphic-program}, lines $284$-$291$. Therefore, there is at least one tuple $a_1,...,a_6 \in \C^{\times}$ for which formula holds, hence, it holds for $\Y_0$. 
\end{proof}

\begin{proposition} \label{th:euler-characteristics-of-Y0}
    Let $\Y_0$ be the special fiber of the family $\Y$, then $\chi(\Y_0) = 192$.
\end{proposition}
\begin{proof}
    As has been discussed before, the decomposition of $\P_{\Pi}$ on toric strata induces a similar decomposition for $\Y_0$
    $$\Y_0 = \bigsqcup_{\sigma \in \Pi} \T_{\sigma} \cap \Y_0,$$
    which leads to the decomposition of the Euler characteristic into a finite sum:
    $$\chi(\Y_0) = \sum_{\sigma \in \Pi} \chi(\T_{\sigma} \cap \Y_0).$$
    By Lemma \ref{th:BKK-lemma} we can use BKK-formula in order to compute the $\chi(\Y_0 \cap \T_{\sigma})$ term:
    $$\chi(\Y_{0} \cap \T_{\sigma}) = (-1)^{5 - \dim \sigma} (5 - \dim \sigma)! \sum_{1 \leq i_1 \leq i_2 \leq i_3 \leq 2} \operatorname{Vol}(\Delta_{1,\sigma}, \Delta_{2,\sigma}, \Delta_{i_1, \sigma}, \Delta_{i_2, \sigma}, \Delta_{i_3, \sigma}).$$
    At this point, one can compute the Euler characteristic with the help of computer algebra systems. See Appendix \ref{ch:appendix-topological-program} for SageMath code.
\end{proof}

\begin{proposition}
    The topological Euler characteristic is $\chi(\Y_{sm}) = 144$, and the holomorphic Euler characteristic is $\chi(\O_{\Y_{sm}}) = 0$. Here, $\Y_{sm}$ denotes a smooth fiber of the family $\Y \to \P^1$. 
\end{proposition}
\begin{proof}
    Both computations follow from the Hodge diamond of the variety $\Y_{sm}$. The Hodge diamond of $\Y_{sm}$ is:
    $$
    \begin{matrix}
      &   &    & 1 &    &   &   \\
      &   & 0  &   & 0  &   &   \\
      & 0 &    & 73 &    & 0 &   \\
    1 &   & 1 &    &  1 &   & 1 \\
      & 0 &    & 73 &    & 0 &   \\
      &   & 0  &   & 0  &   &   \\
      &   &    & 1 &    &   &  
    \end{matrix}.
    $$
    This is proven using the results of \cite{StringyBB}. We refer to the proof of \cite[Proposition 7.3]{Poch1} for an application in a similar setting, reducing to the case of understanding the Hodge diamond of a complete intersection of two cubic hypersurfaces in $\P^5.$
\end{proof}


\begin{lemma}\label{th:holomorphic-euler-characteristics-of-Y0}
    Let $\Y_0 = W_0 \cup W_1 \cup W_2$ is the decomposition of the fiber $\Y_0$ on three irreducible components, described in \cite[Theorem 6.1(8)]{Poch1}. Then $\chi(\O_{W_1})= \chi(\O_{W_2}) = 1$ and $\chi(\O_{W_0}) = 2$.
\end{lemma}
\begin{proof}
    The holomorphic Euler characteristic of a smooth projective variety is  a birational invariant. One conclucdes the first equality since both $W_1$ and $W_2$ are birationally equivalent to $\mathbb{P}^3$, according to \cite[Theorem 6.1(6)]{Poch1}.

The variety $W_0$ is smooth, according to \cite[Theorem 6.1(5)]{Poch1}. Notice that  $W_0 = V(\tilde{h}_1,\tilde{h}_2) = \widetilde{Y}_1 \cap \widetilde{Y}_2 \subset \P_{\Pi}$, where 
$$\tilde{h}_1 = \frac{h_1|_{\psi = 0}}{v_{123}}, \tilde{h}_2 = \frac{h_2|_{\psi = 0}}{u_{456}}, \widetilde{Y}_1 = V(\tilde{h}_1), \widetilde{Y}_2 = V(\tilde{h}_2),$$
and $h_1, h_2$ are defined in \eqref{eq:h1}, \eqref{eq:h2}. From the Koszul resolution of $\mathcal{O}_{W_0}$:
$$
[\mathcal{O}(-\widetilde{Y}_1 - \widetilde{Y}_2) \to \mathcal{O}(-\widetilde{Y}_1) \oplus \mathcal{O}(-\widetilde{Y}_2) \to \mathcal{O}_{\mathbb{P}_{\Pi}}],
$$
We find that:
\begin{equation}\label{eq:holoEulerchar}
\chi(\mathcal{O}_{\widetilde{Y}_1 \cap \widetilde{Y}_2}) = \chi(\mathcal{O}_{\mathbb{P}_{\Pi}}) - \chi(\mathcal{O}(-\widetilde{Y}_1)) - \chi(\mathcal{O}(-\widetilde{Y}_2)) + \chi(\mathcal{O}(-\widetilde{Y}_1 - \widetilde{Y}_2)).
\end{equation}
We will do this by computing the individual terms, using SageMath and \cite[Theorem 13.2.8]{CLS}. For the computation see Appendix \ref{ch:appendix-holomorphic-program}. To use \cite[Theorem 13.2.8]{CLS} we need to find linearly equivalent representatives of divisors $-Y_1 \in \operatorname{Div}(\P_{\Pi}), -Y_2 \in \operatorname{Div}(\P_{\Pi})$ to a sum of prime toric invariant Weil divisors, $D_{u_{ijk}}$ and $D_{v_{ijk}}$, which are in one-to-one correspondence with ray generators $\Pi(1)^{gen}$. Notice that
$$
\widetilde{Y}_1|_{\mathbb{T}_N} = V(t_1 + t_2 + t_3), \widetilde{Y}_2|_{\mathbb{T}_N} = V(t_4 + t_5 + t_6).
$$
We can write down the divisors $L_1, L_2 \in \text{Div}(\P_{\Pi})$ with the prescribed global section spaces 
$$
\Gamma(\P_{\Pi}, \O(L_1)) = \mathbb{C}\{t_1, t_2, t_3\}, \Gamma(\P_{\Pi}, \O(L_2)) = \mathbb{C}\{t_4, t_5, t_6\}.
$$
These are the divisors
$$
L_1 = -D_{v_{123}} + \sum_{\substack{(i,j,k) \neq \\ (1,2,3)}} D_{u_{ijk}}, L_2 = -D_{u_{456}} + \sum_{\substack{(i,j,k) \neq \\ (4,5,6)}} D_{v_{ijk}}.
$$
One can see that $L_1 \sim \widetilde{Y}_1, L_2 \sim \widetilde{Y}_2$. So we will use $-L_1$ and $-L_2$ in computations of holomorphic Euler characteristic. The program written in Appendix \ref{ch:appendix-holomorphic-program} then gives:
$$\chi(\O(-\widetilde{Y}_1)) = 0, \chi(\O(-\widetilde{Y}_2)) = 0, \chi (\O(-\widetilde{Y}_1 -\widetilde{Y}_2)) = 1.$$
 For any smooth complete toric variety $\P_{\Pi}$ we have $\chi(\O_{\P_{\Pi}}) = 1$, and inserting these data into \eqref{eq:holoEulerchar} we find $\chi (\O_{W_0}) = 2.$
\end{proof}

\begin{remark}
    The program in Appendix \ref{ch:appendix-holomorphic-program} uses localization techniques in a non-straightforward manner. If one attempts to compute the localization rational function directly on $t_1, t_2, \dots, t_5$ and then substitutes $t_1 = t_2 = \dots = t_5 = 1$, the computation becomes infeasible. Although we did not perform a detailed asymptotic or complexity analysis, it is evident that this approach is computationally impractical. The key insight is to use a substitution:  
    $$t_1 = 1 + t, \quad t_2 = 1 + t^2, \quad \dots, \quad t_5 = 1 + t^5,$$  
    and then set $t = 0$. This allows the computer to handle rational expressions in a single variable, significantly simplifying the computation and making it feasible.
\end{remark}

\subsection{The family is Kulikov around 0} We consider a degenerating family of Ca\-labi--Yau manifolds $f: \Y \to D$, with $\Y$ being smooth and, in particular, Gorenstein, cf. \cite[Proposition 6.6]{Poch1} and $D$ is a disk centered at $\psi = 0$. Then the relative canonical bundle $K_{\Y/D}$ is a line bundle, which is trivial when restricting to the smooth fibers. We consider the divisor of the evaluation map:
\begin{equation}\label{eq:evaluation}
    f^\ast f_\ast K_{\Y/D} \to K_{\Y/D}. 
\end{equation}
which we denote by $B$. Whenever this is an isomorphism, i.e. $B$ is zero, we say the model is a Kulikov model. Usually this is formulated in terms of the triviality of the canonical bundle of the total space, but this property is implied by the above statement since $f_\ast K_{\Y/D}$ is trivial in a neighborhood of the origin. Since the relative canonical bundle is trivial on the smooth fibers, this map is an isomorphism outside of the origin, and $B$ is concentrated on the special fiber. 
\begin{proposition} \label{th:b-is-zero}
    Let $\Y \to D$ be the mirror family around the origin. Then the family is Kulikov. 
\end{proposition}
\begin{proof}
    First of all, we remind the reader that we use  $\X \to D$ for the family before the desingularization of the ambient toric variety, and $\pi: \Y \to \X$ the natural map. We first show that the corresponding evaluation map \eqref{eq:evaluation} for this family, of the form 
    \begin{equation}\label{eq:evtilde}
        f^\ast f_\ast K_{\X/D} \to K_{\X/D},
    \end{equation}
    is an isomorphism. Indeed, the map is an isomorphism outside the origin as follows from the fact that the canonical sheaves of the fibers $\Y_t$ are free. Moreover, the special fiber $\X_0$ is reduced and irreducible since it is normal by \cite[Proposition 5.4]{Poch1}, and $B$ must hence be a multiple of the special fiber. By the argument of \cite[Section 2.1]{EFiMM} $B$ cannot contain the whole fiber, and must hence be zero. 

    Now, pulling back \eqref{eq:evtilde} via $\pi$ provides an isomorphism 
    \begin{equation}\label{eq:evpullback}
        \pi^{\ast} f^\ast f_\ast K_{\X/D} \to \pi^\ast K_{\X/D}.
    \end{equation}
    Because the resolution of the ambient toric variety is crepant, we have that $\pi^\ast K_{\X} = K_{\Y}$ and since $\X$ is normal we have $\pi_\ast K_{\Y} \simeq  \pi_\ast \pi^\ast K_{\X} \simeq K_{\Y}$. This means that  \eqref{eq:evpullback} has, up to identification, the same source and target as \eqref{eq:evaluation}. By construction, it is the evaluation map outside the origin, and hence it must be the same evaluation map everywhere. But it is an isomorphism so $ B= 0.$
\end{proof}

Now, the formula in Proposition \ref{th:kappa-unipotent}, combined with Proposition \ref{th:holomorphic-euler-characteristics-of-Y0} and Proposition \ref{th:b-is-zero} we find the following: 

\begin{proposition}
    We have $\kappa_0 = 4$.
\end{proposition}

\section{Proof of the main theorem }

The goal of this section is to actually compute the absolute value of the rational function, $|\phi|$. In order to do this we observe that we actually have all the ingredients, we have order of zeros of $\|\eta_p\|_{L^2}$ near the key points and we have a behaviour of BCOV-invariant near the point of interest. In our case $\Delta = \{0,\infty\} \cup \mu_6$. We consider again Theorem \ref{th:ARR} which specializes in our case to the formula:

$$\log \tau_{BCOV} = \log |\phi| + 12 \log \|\eta_0\|^2_{L^2} + \sum_{p=0}^3 (3-p) \log \| \eta_{p} \|_{L^2}^2 + \log C,$$
so that
\begin{equation}\label{eq-delta}
    \log|\phi|= \log \tau_{BCOV} -  12\log \|\eta_0\|_{L^2}^2 - \sum_{p=0}^3 (3-p) \log \| \eta_{p} \|_{L^2}^2 - \log C .
\end{equation}

We know that (some power of) $\phi$ is a rational function on $\P^1$ it is uniquely-up-to-the-constant determined by its divisor. From the construction we know that $\div(\Delta)$ is supported on the set $\{0,\infty\} \cup \mu_6$. We can compute the divisor accurately analyzing left-hand-side of the formula \eqref{eq-delta} in all points except the point $\psi = \infty$, as discussed surrounding \eqref{eq:nidef}. So we have:

\begin{proposition}
    $\log |\phi| $ written in the coordinate $t = \psi - \zeta$ have an expansion
    $$\log |\phi|  = \left(1 + \frac{1}{6}\right) \log |t|^2 + o(\log |t|^2)$$
    in other words $\ord_{\psi = \xi}  |\phi|  = 1 + \frac{1}{6}$.
\end{proposition}
\begin{proof}
    We have to use the formula \eqref{eq-delta}. 
    Incorporating information from computation near $\xi \in \mu_6, t = \psi - \xi$ gives:
    $$\log |\phi| \sim \kappa_\psi \log |t|^2 - 12 \log \|\eta_0 \|_{L^2}^2 -  3 \log \|\eta_{0} \|_{L^2}^2 - 2 \log \|\eta_1 \|_{L^2}^2   - 1 \log \|\eta_2 \|_{L^2}^2  .$$
    we have $\ord_{t=0}(\eta_0) = 0, \ord_{t=0}(\eta_1) = 0, \ord_{t=0}(\eta_2) = -1, \ord_{t=0}(\eta_3) = -1$ combined with equation \eqref{eq:L2growthzeros} and $\kappa_{\psi} = \frac{1}{6}$,  which gives us:
    $$\log |\phi|  \sim \frac{1}{6} \log |t|^2 + 1 \log |t|^2  = \left(1 + \frac{1}{6}\right) \log |t|^2.$$
\end{proof}

Now we have to do the same type of computation near $\psi = 0$:
\begin{proposition}
    $\log |\phi| $ written in the coordinate $\psi$ have an expansion
    $$\log |\phi|  = - 34 \log |\psi|^2 + o(\log |\psi| )$$
\end{proposition}
\begin{proof}
    We have to use the formula \eqref{eq-delta} again. 
    Incorporating information from computation near $\psi = 0$ we have $\kappa_0 = 8$ and $$\ord_{t=0}(\eta_0) = 2, \ord_{t=0}(\eta_1) = 2, \ord_{t=0}(\eta_2) = 4, \ord_{t=0}(\eta_3) = 4.$$ 
    That gives us:
    $$\log |\phi|  \sim 4 \log |\psi|^2 - 24 \log |\psi|^2 - 6 \log |\psi|^2 - 4 \log|\psi|^2 -4 \log|\psi|^2$$
    $$\log |\phi|  \sim -34 \log |\psi|^2$$
\end{proof}
Compiling these results, together with the reformulation of the BCOV conjecture at genus one in Corollary \ref{cor:formulationBCOVgenusone},  one can deduce the following, which amounts to the main Theorem \ref{th-A} :
\begin{theorem} \label{th-tbcov-b}
    For some nonzero constant $C > 0$ we have $|\phi| = C   \left|\psi^{-68} (\psi^6-1)^{7/3}\right|$, hence the BCOV conjecture at genus one holds for the mirror family of intersection of two cubics in $\mathbb{P}^5.$
\end{theorem}

\section{Appendix: Computing topological Euler characteristics of $\Y_0$} \label{ch:appendix-topological-program}
\begin{lstlisting}
# SageMath 10.2

# Importing necessary packages for iterating over combinations
from itertools import combinations_with_replacement, combinations

# The comments for the following functions will use the terminology 
# from the first referenced paper, where the fan \Pi is introduced.

# This function returns the Esd_3 subdivision of the standard 
# 3-dimensional simplex in R^4, dilated by a factor of 3. 
# The simplex is conv{3e_1,...,3e_4}, where e_1,...,e_4 are 
# the standard basis vectors of R^4.
# The function returns a list of simplices, each simplex being 
# a list of points, where each point is a list of non-negative integers.
def get_3_subdivision_of_3_simplex():
    triangles = []
    for i in range(3):
        for j in range(3):
            for k in range(3):
                color_scheme = []
                for t in range(12):
                    r = (t > 4 * i) + (t > 1 + 4 * j) + (t > 2 + 4 * k)
                    color_scheme.append(r)
                triangle = []
                for j1 in range(4):
                    p = [0, 0, 0, 0]
                    for i1 in range(3):
                        p[color_scheme[j1 + i1 * 4]] += 1
                    triangle.append(p)
                triangles.append(triangle)
    return triangles

# This function returns the Esd_3 subdivision of the standard 
# 4-dimensional simplex in R^5, dilated by a factor of 3. 
# The simplex is conv{3e_1,...,3e_5}, where e_1,...,e_5 are 
# the standard basis vectors of R^5.
def get_3_subdivision_of_4_simplex():
    triangles = []
    for i1 in range(3):
        for i2 in range(3):
            for i3 in range(3):
                for i4 in range(3):
                    color_scheme = []
                    for t in range(15):
                        r = (t > 5 * i1) + (t > 1 + 5 * i2) + (t > 2 + 5 * i3) + (t > 3 + 5 * i4)
                        color_scheme.append(r)
                    triangle = []
                    for j in range(5):
                        p = [0, 0, 0, 0, 0]
                        for i in range(3):
                            p[color_scheme[j + i * 5]] += 1
                        triangle.append(p)
                    triangles.append(triangle)
    return triangles

# This function returns the subdivision of the prism over sigma and sigma+w, 
# where w = (1,1,1,-1,-1,-1)^T.
# The function returns a list of simplices, each represented as a list of points 
# (each point being a list of non-negative integers).
def get_prism_subdivision(sigma):
    delta = deepcopy(sigma)
    for point in delta:
        point[0] += 1
        point[1] += 1
        point[2] += 1
        point[3] -= 1
        point[4] -= 1
        point[5] -= 1
    triangles = []
    for i in range(4):
        triangle = []
        for t in range(i + 1):
            triangle.append(sigma[t])
        for t in range(i, 4):
            triangle.append(delta[t])
        triangles.append(triangle)
    return triangles

# This function returns the list of maximal cones of the fan Pi. 
# Each cone is a list of ray generators (lists of integers).
def get_maximal_cones_of_Pi():
    cones = []
    for i in range(3):
        triangles = get_3_subdivision_of_4_simplex()
        for triangle in triangles:
            for point in triangle:
                point.insert(i, 0)
                point[0] -= 1
                point[1] -= 1
                point[2] -= 1
        cones += triangles
    for i in range(3, 6):
        triangles = get_3_subdivision_of_4_simplex()
        for triangle in triangles:
            for point in triangle:
                point.insert(i, 0)
                point[3] -= 1
                point[4] -= 1
                point[5] -= 1
        cones += triangles
    for i in range(3):
        for j in range(3, 6):
            triangles = get_3_subdivision_of_3_simplex()
            for triangle in triangles:
                for point in triangle:
                    point.insert(i, 0)
                    point.insert(j, 0)
                    point[0] -= 1
                    point[1] -= 1
                    point[2] -= 1
                cones += get_prism_subdivision(triangle)
    return cones

# This function constructs the fan Pi as an instance of the 
# sage.geometry.fan.RationalPolyhedralFan class.
def get_fan_Pi():
    # Defining the polytope P, which has 12 vertices and 111 integer points, including the origin.
    vertices = [
        (3, 0, 0, -1, -1, -1),
        (0, 3, 0, -1, -1, -1),
        (0, 0, 3, -1, -1, -1),
        (0, 0, 0, 2, -1, -1),
        (0, 0, 0, -1, 2, -1),
        (0, 0, 0, -1, -1, 2),
        (2, -1, -1, 0, 0, 0),
        (-1, 2, -1, 0, 0, 0),
        (-1, -1, 2, 0, 0, 0),
        (-1, -1, -1, 3, 0, 0),
        (-1, -1, -1, 0, 3, 0),
        (-1, -1, -1, 0, 0, 3)
    ]
    P = Polyhedron(vertices)
    
    # Integral points of P (which correspond to ray generators of Pi, excluding the origin).
    integral_points = P.integral_points()
    rays_of_Pi = [p for p in integral_points if p != vector([0] * P.ambient_dim())]
    
    nonreduced_cones = get_maximal_cones_of_Pi()
    sage_cones = []
    
    for cone_generators in nonreduced_cones:
        rays = [vector(ray) for ray in cone_generators]
        cone = Cone(rays)
        sage_cones.append(cone)
    
    return Fan(sage_cones)
#This function, takes 3 natural numbers i,j,k and 
#returns a list of integers, which is a generator u_{ijk} \in \Pi(1)^{gen}
#notice that u(1,2,3) is not a valid generator!
def u(i,j,k):
    a=[-1,-1,-1,0,0,0]
    a[i-1]+=1
    a[j-1]+=1
    a[k-1]+=1
    return a

#This function, takes 3 natural numbers i,j,k and 
#returns a list of integers, which is a generator v_{ijk} \in \Pi(1)^{gen}, represented by a list of integers 
#notice that v(4,5,6) is not a valid generator!
def v(i,j,k):
    a=[0,0,0,-1,-1,-1]
    a[i-1]+=1
    a[j-1]+=1
    a[k-1]+=1
    return a

# Kronecker delta function
def delta(i, j):
    return 1 if i == j else 0

# This is a ring Q[x1,...,x5] where x1,...,x5 will represent a canonical 
# coordinates on the affine patch spec C[sigma^{vee} \cap M]. The fact that
# we allow to use only rational coefficients instead of complex doesn't play
# a role.
R.<x1, x2, x3, x4, x5> = PolynomialRing(QQ, 5)

# Takes a cone sigma and a fan Pi,
# returns some maximal cone delta of Pi which contains sigma.
def find_5_cone_containing_sigma(Pi, sigma):
    for cone in Pi:
        if cone.dimension() == 5 and sigma.is_face_of(cone):
            return cone
    return None

# Newton polytope of polynomial,
# takes polynomial on variables x1,...,x5 and returns Newton polytope (which naturally lives in R^5).
def newton_polytope(polynomial):
    exponents = [m.exponents()[0] for m in polynomial.monomials()]
    return Polyhedron(vertices = exponents)

# The function takes list of equations and a natural number and returns an integer,
# this is one term in a chi(V(h1,h2)|_T_{sigma}) = sum_{1 <= i1 <= ... <= i3 <= 2} Vol(Delta1, Delta2, ..., Deltai3) formula.
def compute_mixed_volume(equations, dim):
    n = len(equations)
    equations = [p for p in equations]
    newton_polytopes = [newton_polytope(p) for p in equations]
    vol = 0;
    for d in range(1, n + 1):
        for subset in combinations(newton_polytopes, d):
            k = len(subset)
            P = reduce(lambda x, y: x + y, subset)
            if P.dim() == dim:
                vol += (-1)^(n-k) * P.volume(measure='induced')
    return vol
    
# The function computes Euler characteristic of T_sigma \cap V(equations) through mixed volumes
# it takes equations, dimension of the cone sigma and returns an integer.
def compute_chi_by_equations(equations, dim):
    equations = [eq for eq in equations if eq != 0]
    if contains_nonzero_degree_0(equations):
        return 0
    chi = 0
    for comb in combinations_with_replacement(equations, dim-len(equations)):
        vol = compute_mixed_volume(equations + list(comb), dim)
        chi += vol
    return (-1)^(dim-len(equations))*chi

# The function takes set of polynomials and returns True if there is a nonzero degree 0 equation
# and false otherwise (for technical reasons) this special case should be treated separately.
def contains_nonzero_degree_0(polynomials):
    for poly in polynomials:
        poly *= x1^0
        if (poly).degree() == 0 and poly != 0:
            return True
    return False

# Computes one term of Euler characteristics in inclusion-exclusion formula, 
# takes fan and cone as an input and returns number as an output.
def compute_chi(Pi, cone):
    supercone = find_5_cone_containing_sigma(Pi, cone)
    mono = [1,1,1,1,1,1]
    x = [x1,x2,x3,x4,x5]
    cone = [list(ray) for ray in cone.rays()]
    supercone = [list(ray) for ray in supercone.rays()]

    
    for i in range(1,7):
        for j in range(i,7):
            for k in range(j,7):
                for t in range(1,7):
                    if u(i,j,k) in supercone:
                        p = delta(i,t)+delta(j,t)+delta(k,t)
                        ind = supercone.index(u(i,j,k))
                        mono[t-1] *= x[ind]^(p)
                    if v(i,j,k) in supercone:
                        p = delta(i,t)+delta(j,t)+delta(k,t)
                        ind = supercone.index(v(i,j,k))
                        mono[t-1] *= x[ind]^(p)
    
    allu = 1
    allv = 1
    for i in range(1,7):
        for j in range(i,7):
            for k in range(j,7):
                    if u(i,j,k) in supercone:
                        ind = supercone.index(u(i,j,k))
                        allu *= x[ind]
                    if v(i,j,k) in supercone:
                        ind = supercone.index(v(i,j,k))
                        allv *= x[ind]

    nonzero_variables = []
    for vec in supercone:
        if vec not in cone:
            ind = supercone.index(vec)
            nonzero_variables.append(x[ind])
    
    for vec in supercone:
        if vec in cone:
            ind = supercone.index(vec)
            if x[ind].divides(allu): 
                    allu = 0
            if x[ind].divides(allv): 
                    allv = 0
            for t in range(6):
                if x[ind].divides(mono[t]): 
                    mono[t] = 0
    
    h1 = - mono[1-1] - mono[2-1] - mono[3-1]
    h2 = - mono[4-1] - mono[5-1] - mono[6-1]
    equations = [h1, h2]
    equations = [eq for eq in equations if eq != 0]

    # This snippet verifies that the Newton polytopes Delta_1, Delta_2 have no integral points
    # other than those corresponding to the monomials of h_1 and h_2.
    for eq in equations:
        if len(newton_polytope(eq).integral_points()) != len(eq.monomials()):
            print(newton_polytope(eq).integral_points())
            print(eq.monomials())
            print('Counterexample is found')
            
    return compute_chi_by_equations(equations, 5-len(cone))

# Computes Euler characteristic of \Y_0 \cap \T_N.
Pi = get_fan_Pi()
chi = 0
for d in range(0, 6):
    for cone in Pi.cones(d):
        r = compute_chi(Pi, cone)
        chi += r
print("Euler characteristic:", chi)
# Works approximately 10 minutes,
# prints "Euler characteristic: 192" in the output.
\end{lstlisting}

\section{Appendix: Computing holomorphic Euler characteristics of $\mathcal O_{\Y_0}$} \label{ch:appendix-holomorphic-program}
\begin{lstlisting}
#SageMath 10.2

from sage.schemes.toric.variety import ToricVariety
from sage.rings.integer_ring import ZZ
from collections import Counter
from copy import deepcopy

# The next 6 functions do the same thing that they do in the previous program
# namely, get_fan_Pi() returns the fan Pi, object of the sage.geometry.fan.RationalPolyhedralFan class,
# the only difference is that the ray generators now in the lattice Z^5, not in the N \subset Z^6,
# it is more convinient since programmatically it is harder to work with quotient lattice M = Z^6/Z{(1,...,1)} than with Z^5.
def get_3_subdivision_of_3_simplex():
    triangles = []
    for i in range(3):
        for j in range(3):
            for k in range(3):
                color_scheme = []
                for t in range(12):
                    r = (t > 4 * i) + (t > 1 + 4 * j) + (t > 2 + 4 * k)
                    color_scheme.append(r)
                triangle = []
                for j1 in range(4):
                    p = [0, 0, 0, 0]
                    for i1 in range(3):
                        p[color_scheme[j1 + i1 * 4]] += 1
                    triangle.append(p)
                triangles.append(triangle)
    return triangles


def get_3_subdivision_of_4_simplex():
    triangles = []
    for i1 in range(3):
        for i2 in range(3):
            for i3 in range(3):
                for i4 in range(3):
                    color_scheme = []
                    for t in range(15):
                        r = (t > 5 * i1) + (t > 1 + 5 * i2) + (t > 2 + 5 * i3) + (t > 3 + 5 * i4)
                        color_scheme.append(r)
                    triangle = []
                    for j in range(5):
                        p = [0, 0, 0, 0, 0]
                        for i in range(3):
                            p[color_scheme[j + i * 5]] += 1
                        triangle.append(p)
                    triangles.append(triangle)
    return triangles


def get_prism_subdivision(sigma):
    delta = deepcopy(sigma)
    for point in delta:
        point[0] += 1
        point[1] += 1
        point[2] += 1
        point[3] += -1
        point[4] += -1
        point[5] += -1
    triangles = []
    for i in range(4):
        triangle = []
        for t in range(i + 1):
            triangle.append(sigma[t])
        for t in range(i, 4):
            triangle.append(delta[t])
        triangles.append(triangle)
    return triangles


def get_maximal_cones_of_Pi():
    cones = []
    for i in range(3):
        triangles = get_3_subdivision_of_4_simplex()
        for triangle in triangles:
            for point in triangle:
                point.insert(i, 0)
                point[0] += -1
                point[1] += -1
                point[2] += -1
        cones += triangles
    for i in range(3, 6):
        triangles = get_3_subdivision_of_4_simplex()
        for triangle in triangles:
            for point in triangle:
                point.insert(i, 0)
                point[3] += -1
                point[4] += -1
                point[5] += -1
        cones += triangles
    for i in range(3):
        for j in range(3, 6):
            triangles = get_3_subdivision_of_3_simplex()
            for triangle in triangles:
                for point in triangle:
                    point.insert(i, 0)
                    point.insert(j, 0)
                    point[0] += -1
                    point[1] += -1
                    point[2] += -1
                cones += get_prism_subdivision(triangle)
    return cones


def get_rays_of_Pi():
    vertices = [(3, 0, 0, -1, -1),
                (0, 3, 0, -1, -1),
                (0, 0, 3, -1, -1),
                (0, 0, 0, 2, -1),
                (0, 0, 0, -1, 2),
                (0, 0, 0, -1, -1),
                (2, -1, -1, 0, 0),
                (-1, 2, -1, 0, 0),
                (-1, -1, 2, 0, 0),
                (-1, -1, -1, 3, 0),
                (-1, -1, -1, 0, 3),
                (-1, -1, -1, 0, 0)]
    P = Polyhedron(vertices)
    integral_points = P.integral_points()
    rays_of_Pi = [p for p in integral_points if p != vector([0] * P.ambient_dim())]
    return rays_of_Pi


def get_fan_Pi():
    vertices = [(3, 0, 0, -1, -1),
                (0, 3, 0, -1, -1),
                (0, 0, 3, -1, -1),
                (0, 0, 0, 2, -1),
                (0, 0, 0, -1, 2),
                (0, 0, 0, -1, -1),
                (2, -1, -1, 0, 0),
                (-1, 2, -1, 0, 0),
                (-1, -1, 2, 0, 0),
                (-1, -1, -1, 3, 0),
                (-1, -1, -1, 0, 3),
                (-1, -1, -1, 0, 0)]
    P = Polyhedron(vertices)
    integral_points = P.integral_points()
    rays_of_Pi = [p for p in integral_points if p != vector([0] * P.ambient_dim())]
    nonreduced_cones = get_maximal_cones_of_Pi()
    sage_cones = []
    for cone_generators in nonreduced_cones:
        rays = [vector(ray)[:-1] for ray in cone_generators]
        cone = Cone(rays)
        sage_cones.append(cone)
    return Fan(sage_cones)


# This is polynomial ring Q[t], the fact that coefficients are in Q, not in C doesn't play a role.
R.<t> = PolynomialRing(QQ, 1)

# This is a function which takes a list and returns monomial, 
# in the formula written in [CLS, Theorem 13.2.8] there is monomial on 5 variables,
# x1,x2,x3,x4,x5. We reduce each variable x1=(1+t), x2=(1+t^2), ..., x5=(1+t^5).
# The value of r(x1(t),...,x5(t)) at t=0 is equal to value of r(x1,...,x5) at x1=...=x5=1,
# if the former function is computable (have no pole) at t=0. This trick significantly reduces
# the computational time.
def monomial_by_vector(vec):
    return ((1 + t) ** vec[0]) * ((1 + t ** 2) ** vec[1]) * ((1 + t ** 3) ** vec[2]) * ((1 + t ** 4) ** vec[3]) * \
           ((1 + t ** 5) ** vec[4])

# Takes a maximal cone, takes a divisor D and returns S(chi-tilde(U_sigma,O(D))) (in a reduced form in which xi = (1+t^i)),
# in terms of [CLS, Theorem 13.2.8].
def compute_holo_chi_piece(cone, D):
    cone_list = [list(ray) for ray in cone.rays()]
    
    A = Matrix(cone_list)
    A_inv = A.inverse()
    b = vector([0] * 5)
    
    for i in range(5):
        row_vector = tuple(A.row(i))
        b[i] = D[row_vector]
    m_sigma = -A_inv * b
    en = monomial_by_vector(m_sigma)

    den = 1
    cone_dual = cone.dual()
    cone_dual_list = [list(ray) for ray in cone_dual.rays()]
    for i in range(5):
        den *= (1 - monomial_by_vector(cone_dual_list[i]))

    return en / den

Pi = get_fan_Pi()

# Initialization of divisors. 
L1, L2, L12 = Counter(), Counter(), Counter()

for ray in Pi.rays():
    L1[tuple(ray)] = L2[tuple(ray)] = L12[tuple(ray)] = 0

for ray in Pi.rays():
    if ray[0] == -1 or ray[1] == -1 or ray[2] == -1:
        L1[tuple(ray)] -= 1
        L12[tuple(ray)] -= 1

for ray in Pi.rays():
    if ray[0] >= 0 and ray[1] >= 0 and ray[2] >= 0:
        L2[tuple(ray)] -= 1
        L12[tuple(ray)] -= 1

u456 = tuple(u(4, 5, 6))
v123 = tuple(v(1, 2, 3))
L1[v123] += 1
L2[u456] += 1
L12[v123] += 1
L12[u456] += 1

#The main computation, direct summation via formula in [CLS, Theorem 13.2.8].
r12 = r1 = r2 = 0
for cone in Pi.cones(5):
    r12 += compute_holo_chi_piece(cone, L12)
    r1 += compute_holo_chi_piece(cone, L1)
    r2 += compute_holo_chi_piece(cone, L2)

print("Euler characteristic of O(-Y1-Y2): ", r12(0))
print("Euler characteristic of O(-Y1): ", r1(0))
print("Euler characteristic of O(-Y2): ", r2(0))

# The program outputs "1, 0, 0" after approximately 10 minutes.
\end{lstlisting}

\bibliographystyle{plain}
\bibliography{Biblio}{}

\begin{thebibliography}{10}

\bibitem{MPCP}
V.~Batyrev.
\newblock Dual polyhedra and mirror symmetry for {C}alabi-{Y}au hypersurfaces in toric varieties.
\newblock {\em J. Algebraic Geom.}, 3(3):493--535, 1994.

\bibitem{StringyBB}
V.~Batyrev and L.~Borisov.
\newblock Mirror duality and string-theoretic {H}odge numbers.
\newblock {\em Invent. Math.}, 126(1):183--203, 1996.

\bibitem{BB}
V.~Batyrev and L.~Borisov.
\newblock On {C}alabi-{Y}au complete intersections in toric varieties.
\newblock In {\em Higher-dimensional complex varieties ({T}rento, 1994)}, pages 39--65. de Gruyter, Berlin, 1996.

\bibitem{BvS}
V.~Batyrev and D.~van Straten.
\newblock Generalized hypergeometric functions and rational curves on {C}alabi-{Y}au complete intersections in toric varieties.
\newblock {\em Comm. Math. Phys.}, 168(3):493--533, 1995.

\bibitem{BernsteinHovanski}
D.~N. Bernstein, A.~G. Ku{\v{s}}nirenko, and A.~G. Hovanski{\u{i}}.
\newblock Newton polyhedra.
\newblock {\em Uspehi Mat. Nauk}, 31(3(189)):201--202, 1976.

\bibitem{BCOV}
M.~Bershadsky, S.~Cecotti, H.~Ooguri, and C.~Vafa.
\newblock Kodaira-{S}pencer theory of gravity and exact results for quantum string amplitudes.
\newblock {\em Comm. Math. Phys.}, 165(2):311--427, 1994.

\bibitem{CdGP}
P.~Candelas, X.~de~la Ossa, P.~Green, and L.~Parkes.
\newblock A pair of {C}alabi-{Y}au manifolds as an exactly soluble superconformal theory.
\newblock {\em Nuclear Phys. B}, 359(1):21--74, 1991.

\bibitem{CL}
E.~Coddington and N.~Levinson.
\newblock {\em Theory of ordinary differential equations}.
\newblock McGraw-Hill Book Co., Inc., New York-Toronto-London, 1955.

\bibitem{CK}
D.~Cox and S.~Katz.
\newblock {\em Mirror symmetry and algebraic geometry}, volume~68 of {\em Mathematical Surveys and Monographs}.
\newblock American Mathematical Society, Providence, RI, 1999.

\bibitem{CLS}
D.~A. Cox, J.~Little, and H.~Schenck.
\newblock {\em Toric varieties}, volume 124 of {\em Graduate Studies in Mathematics}.
\newblock American Mathematical Society, Providence, RI, 2011.

\bibitem{EFiMM}
D.~Eriksson, G.~Freixas~i Montplet, and C.~Mourougane.
\newblock Singularities of metrics on {H}odge bundles and their topological invariants.
\newblock {\em Algebr. Geom.}, 5(6):742--775, 2018.

\bibitem{EFiMM2}
D.~Eriksson, G.~Freixas~i Montplet, and C.~Mourougane.
\newblock B{COV} invariants of {C}alabi-{Y}au manifolds and degenerations of {H}odge structures.
\newblock {\em Duke Math. J.}, 170(3):379--454, 2021.

\bibitem{EFiMM3}
D.~Eriksson, G.~Freixas~i Montplet, and C.~Mourougane.
\newblock On genus one mirror symmetry in higher dimensions and the {BCOV} conjectures.
\newblock {\em Forum Math. Pi}, 10:Paper No. e19, 53, 2022.

\bibitem{FLY}
H.~Fang, Z.~Lu, and K.-I. Yoshikawa.
\newblock Analytic torsion for {C}alabi-{Y}au threefolds.
\newblock {\em J. Differential Geom.}, 80(2):175--259, 2008.

\bibitem{givental}
A.~Givental.
\newblock A mirror theorem for toric complete intersections.
\newblock In {\em Topological field theory, primitive forms and related topics ({K}yoto, 1996)}, volume 160 of {\em Progr. Math.}, pages 141--175. Birkh\"{a}user Boston, Boston, MA, 1998.

\bibitem{malter}
A.~Malter.
\newblock A derived equivalence of the {L}ibgober-{T}eitelbaum and the {B}atyrev-{B}orisov mirror constructions, 2023.

\bibitem{Mor}
D.~Morrison.
\newblock Picard-{F}uchs equations and mirror maps for hypersurfaces.
\newblock In {\em Essays on mirror manifolds}, pages 241--264. Int. Press, Hong Kong, 1992.

\bibitem{Poch1}
M.~Pochekai.
\newblock Geometry of the mirror models dual to the complete intersection of two cubics, 2023.
\newblock \href{https://arxiv.org/abs/2311.15103}{arXiv:2311.15103}.

\bibitem{Popa}
A.~Popa.
\newblock The genus one {G}romov-{W}itten invariants of {C}alabi-{Y}au complete intersections.
\newblock {\em Trans. Amer. Math. Soc.}, 365(3):1149--1181, 2013.

\bibitem{rossi}
M.~Rossi.
\newblock Non-calibrated framed processes, derived equivalence and homological mirror symmetry, 2023.

\bibitem{Schmid}
W.~Schmid.
\newblock Variation of {H}odge structure: the singularities of the period mapping.
\newblock {\em Invent. Math.}, 22:211--319, 1973.

\bibitem{Ste}
J.~H.~M. Steenbrink.
\newblock Mixed {H}odge structure on the vanishing cohomology.
\newblock In {\em Real and complex singularities ({P}roc. {N}inth {N}ordic {S}ummer {S}chool/{NAVF} {S}ympos. {M}ath., {O}slo, 1976)}, pages 525--563. Sijthoff \& Noordhoff, Alphen aan den Rijn, 1977.

\bibitem{ZagierZinger}
D.~Zagier and A.~Zinger.
\newblock Some properties of hypergeometric series associated with mirror symmetry.
\newblock In {\em Modular forms and string duality}, volume~54 of {\em Fields Inst. Commun.}, pages 163--177. Amer. Math. Soc., Providence, RI, 2008.

\bibitem{Zin}
A.~Zinger.
\newblock Reduced genus-one {G}romov-{W}itten invariants.
\newblock {\em J. Differential Geom.}, 83(2):407--460, 2009.

\end{thebibliography}
\end{document}